\documentclass[10pt,a4paper]{article}
\usepackage[T1]{fontenc}
\usepackage{amsmath, amsthm, amsfonts, amssymb, enumerate, xcolor,bm,graphicx, url}
\usepackage[margin=1in]{geometry}

\usepackage{algpseudocode,algorithm}

\newcommand{\imgscale}{0.75}

\newtheorem{theorem}{Theorem}

\newtheorem{proposition}[theorem]{Proposition}
\newtheorem{lemma}[theorem]{Lemma}
\newtheorem{assumption}{Assumption}
\numberwithin{theorem}{section}
\renewcommand{\d}{\mathrm{d}}
\newcommand{\C}{\mathbb{C}}
\newcommand{\R}{\mathbb{R}}
\newcommand{\Z}{\mathbb{Z}}

\newcommand{\bE}{\mathbb{E}}
\newcommand{\bV}{\mathbb{V}}
\newcommand{\bP}{\mathbb{P}}

\newcommand{\psib}{\bm{\psi}}
\renewcommand{\S}{\mathcal{S}}
\newcommand{\V}{\mathcal{V}}

\newcommand{\bv}{\mathbf{v}}
\newcommand{\bw}{\mathbf{w}}
\newcommand{\bl}{\mathbf{l}}
\newcommand{\br}{\mathbf{r}}
\newcommand{\sym}{{\mathrm{sym}}}

\newcommand{\pX}{\psib_{X}}
\newcommand{\pY}{\psib_{Y}}

\newcommand{\PXX}{\Psi_{XX}}
\newcommand{\PXY}{\Psi_{XY}}
\newcommand{\hPXX}{\hat{\Psi}_{XX}}
\newcommand{\hPXY}{\hat{\Psi}_{XY}}

\newcommand{\Cl}[1][\lambda]{C_{#1}}
\newcommand{\cl}[2][\lambda]{C_{#1,#2}}
\newcommand{\hCl}[1][\lambda]{\hat{C}_{#1}}
\newcommand{\tCl}[1][\lambda]{\tilde{C}_{#1}}

	\newcommand{\Qst}{Q_{\lambda}}

\DeclareMathOperator{\tr}{tr}
\DeclareMathOperator{\specrad}{spec{\,}rad}
\DeclareMathOperator{\Ces}{Ces}

\DeclareMathOperator{\interior}{int}
\DeclareMathOperator{\im}{im}

\DeclareMathOperator*{\esssup}{ess{\,}sup}

\newcommand{\iid}{{\em i.i.d{.}}}

\title{Sampling pseudospectrum for data-driven matrices}
\author{Caroline Wormell\thanks{School of Mathematics and Statistics, The University of Sydney \\
		email: {\sf caroline.wormell@sydney.edu.au}\\
		ORCID: 0000-0003-2953-6493}}
\begin{document}
\maketitle

\begin{abstract}
	Many complex systems can be reduced to their key components through spectrally decomposing matrices that capture their dynamics. These matrices can in turn be constructed from data, often by least-squares fitting: examples of algorithms to do this include Dynamical Mode Decomposition and variants, subspace identification and eigenvalue realisation algorithms. Typical outputs of these algorithms include a range of isolated, peripheral eigenvalues capturing persistent emergent patterns in the system. However, there is no objective way to assess which of these discrete eigenvalues are artefacts of finite data error, and which are reflections of a fully sampled operator. 
	
	In this paper, we present a sampling pseudospectrum $P(\lambda)$, that provides probabilistic information on the behaviour of finite-data eigenvalues in the complex plane, and an estimator $\hat P(\lambda)$, which can be obtained by reprocessing our finite data sample. The estimator, which is computationally efficient to implement, allows us to test statistically for the location of the true eigenvalues. This gives us a rigorous and very general way to assess whether the patterns we extract from finite data are likely to be signal or noise.
\end{abstract}

Large systems, from turbulent fluids to power grids to economies, often behave in ways that are not immediately evident from the equations that govern them, which themselves might not be known. Instead, they are very often studied by constructing a linear operator whose spectrum reflects properties of the original system. This linear operator is often inferred from data. In dynamical systems, this is the philosophy behind Dynamic Mode Decomposition and its variants, including Extended Dynamic Mode Decomposition (EDMD) and kernel EDMD; in systems theory it appears in methods such as subspace identification and eigensystem realisation algorithms
\cite{Schmid22,Williams15,Korda18,Overschee96,Juang85}.
In each case, a finite set of observations is used to construct a matrix, or a generalised eigenvalue problem, whose eigenvalues and eigenvectors are then interpreted as persistent modes, resonances, coherent structures or metastable patterns of the underlying system.

The difficulty is that these matrices are themselves random objects, because they come from a finite data sample. These matrices can be understood as approximations to a ``true'' matrix that that would be obtained with infinitely many samples.  Since dynamical operators, and therefore matrices arising in data-driven operator approximation, are often highly non-normal, their eigenvalues can be acutely and differentially sensitive to perturbations. It is therefore not immediately clear how to assess which of a large cloud of eigenvalues contain meaningful spectral information, and which are artefacts of finite sampling error, akin to eigenvalues of a random matrix. 
In applications, this poses a basic question: which isolated eigenvalues should be trusted?

We can illustrate this problem in the context of Koopman-based methods.  Koopman spectral analysis seeks to describe nonlinear dynamics through the linear action of the {\it Koopman operator} $\mathcal{K}$ on observable functions $\varphi$ of the system's state \cite{Koopman31,Mezic05}, defined by $\mathcal{K} \varphi = \varphi \circ f$, where $f$ is the time-one map of the dynamical system.  EDMD and related least-squares methods provide computationally tractable approximations of this operator, when restricted to a finite dictionary of observables $\varphi$.  For chaotic systems, the spectral objects of interest are often Ruelle--Pollicott resonances (or finite-dimensional estimates of them) which govern rates of decay, oscillatory
components and coherent behaviour.  However, a priori error analysis of these resonances is difficult outside of toy models \cite{Herwig25}, as the resonances are typically only discrete eigenvalues of the Koopman operator when posed on Banach spaces that are difficult to identify explicitly. Furthermore, it is not clear how to accurately assess the spectral error associated to data-based discretisations. Recent work in this area has studied $L^2$ spectra, which in the chaotic setting are continuous and supported on the unit circle, and require coercion to unitary operators to be evidenced in DMD and EDMD \cite{Colbrook22}.

Classical pseudospectra provide one response to eigenvalue sensitivity: one chooses a matrix norm, and asks where eigenvalues could move under matrix perturbations of a given size in this norm
\cite{Trefethen05}. This paper develops an analogous object, specifically geared to sampling error. Rather than considering arbitrary worst-case perturbations in a fixed norm, we consider the random perturbations likely to produced by replacing an expectation with a finite sample average. Furthermore, because the appropriate norm is rarely known in applications, particularly in Koopman analysis, we optimise over all Hilbert norms. This provides a norm-independent output. The resulting object is a
\emph{sampling pseudospectrum}: a function $P$ on the complex plane that quantifies
how plausible it is for data-driven eigenvalues to occur near a given point.

The sampling pseudospectrum has a range of pseudospectrum-like properties, but has a probabilistic flavour. The key point is that it approximately bounds the log-likelihood of a finite-data eigenvalue in a particular set (Theorem~\ref{t:PIsSpecRad}): in particular,
\[ \bP(\lambda \textrm{ is close to an eigenvalue}) \lesssim e^{- M P(\lambda) / 2}, \]
with $P(\lambda) = 0$ if and only if $\lambda$ is an eigenvalue of the infinite-data limit (Proposition~\ref{p:PvsEigvals}). Similar to the classical pseudospectrum, we this allows us to cluster eigenvalues, with an eigenvalue counting result holds within level sets of $P(\lambda)$ (Theorem~\ref{t:PIsSpecRad}). This provides an objective tool to study how the spectra of Koopman and other operator approximations respond when finite-data error is introduced.

The key practical issue, however, is to work the other way: to use finite data to study infinite-data spectra. We construct an estimator $\hat P(\lambda)$ that allows us to do this. There are two key facts about $\hat P(\lambda)$. The first is that it is a consistent estimator of $P(\lambda)$ as the number of snapshots $M \to \infty$ (note that in this paper we assume that the observable dictionary is fixed). The second is that if $\hat P(\lambda) \to 0$ as $M\to\infty$, i.e. if $\lambda$ is a true eigenvalue, then $M \hat P(\lambda)$ has a certain distributional limit that can be used to test statistically for $\lambda$ being a true eigenvalue. This can be used to rule true eigenvalues out of areas of the phase space from finite data, and thus group finite-data eigenvalues into clusters that are likely to remain spectrally independent as the amount of data is increased. We have implemented methods to compute $\hat P(\lambda)$ in a Julia package {\tt SamplingPseudospectrum.jl} \cite{SP}

Both $P(\lambda)$ and $\hat P(\lambda)$ are formulated as optimisation problems over Hilbert norms (see Section~\ref{ss:Indicator} and Section~\ref{s:StatisticalTest} respectively). It turns out that these optimisation problems can be effectively represented as spectral radii of operators on matrices, and bounded via the power method (see Section~\ref{s:Algorithms}). This makes computation of sampling pseudospectra simple and easy: in practice, computing $\hat P(\lambda)$ for a given $\lambda$ takes the same order of time as computing the least-squares matrix approximation. 

The approach is highly general, and can be used to solve a large array of generalised eigenvalue problems. The main requirement on the kinds of data sampling that it can study is that it has a central limit theorem: this includes all kinds of random sampling, as well as ergodic sampling by chaotic systems.

The paper is structured as follows. In Section~\ref{s:Setup} we set up the problem and provide the mathematical notation used in the sequel. In Section~\ref{s:GeneralConcept} we present the basic idea underpinning the quantities we then present: the theoretical sampling pseudospectrum, giving probabilistic bounds on the location of the data-driven Koopman spectrum (Section~\ref{s:ProbabilisticBounds}) and the data-based estimator of it, which provides a statistical test for eigenvalue location (Section~\ref{s:StatisticalTest}). We then present efficient algorithms to compute the relevant quantities (Sections~\ref{s:Algorithms}--\ref{s:Variance}). Finally, we evaluate our methods on some examples: an autoregressive process (Section~\ref{s:VAR}), 1D expanding dynamics (Section~\ref{s:Numerics1D}), the Lorenz attractor (Section~\ref{s:NumericsL63}) and Rayleigh--B\'enard convection (Section~\ref{s:NumericsRB}). Finally we discuss future directions in Section~\ref{s:Conclusion}.

\section{Problem setup and notation}\label{s:Setup}

We begin by presenting this problem abstractly (and very generally), and then explain how the standard Dynamical Mode Decomposition problem fits into this framework.

\subsection{Abstract formulation}\label{ss:AbstractFormulation}

We will model our sampling process as a random dynamical system $(\Omega, \theta, \mu)$, as described in \cite{Arnold98}. (This dynamical system may be different to any dynamics associated to the operator, and can cover \iid and other random sampling.) The state of our sampling process at time $m$ is given as an event $\omega_m$ in a state space $\Omega$. This event may contain hidden information, including about future states, allowing for random sampling in this framework. The state evolves ergodically as $\omega_{m+1} = \theta(\omega_m)$, with invariant measure $\mu$. 

From each $\omega \in \Omega$ and for each $\lambda \in E \subseteq \C$, we can observe some snapshot ``characteristic matrix'' \begin{align}
	\cl{\omega} \in \C^{N\times N} \tag{snapshot characteristic matrix}
\end{align}
We assume these matrices are uniformly bounded in $\lambda, \omega$, and analytic functions of $\lambda$ on $E$.

Let us define the expectation and sample average of these matrices:
\begin{align} \Cl &:= \bE[\cl{\omega}] \tag{characteristic matrix} \\
\hCl &:= \frac{1}{M} \sum_{m=1}^M \cl{\omega_m}. \tag{sample characteristic matrix}
\end{align}

We will notate the difference between these two
\begin{align*} \tCl &:= \hCl - \Cl \end{align*}
and we can see that $\bE[\tCl] = 0$. 

We can imagine these matrices as characteristic matrices of generalised eigenvalue problems, with a notion of ``spectrum'' given by
\begin{align}  \Sigma &= \{ \lambda \in E : \ker \Cl \neq \{0\} \} \tag{generalised spectrum} \\
\hat \Sigma &= \{ \lambda \in E : \ker \hCl \neq \{0\} \} \tag{sample generalised spectrum}
\end{align}
In both cases, the geometric multiplicity of $\lambda \in \Sigma$ is the multiplicity of the relevant kernel; the algebraic multiplicity may not match that of the kernel, but is the rank of the Riesz projection 
$ \frac{1}{2\pi i} \int_{\Gamma} \Cl[z]^{-1} \d z $
for some small closed circle $\Gamma$ around $\lambda$.

The aim of this paper is to provide a way to learn as much as we can about one set of these generalised eigenvalues using only information associated with the other.

\subsection{Koopman context}\label{ss:KoopmanContext}

As a significant example, let us translate this setup into the context of Koopman theory.

In the Koopman operator context, our samples $\omega$ yield realisations $(X(\omega),Y(\omega))$ where $X(\omega), Y(\omega)$ are in the phase space $\Lambda$ of the dynamical system $f:\Lambda \to \Lambda$ corresponding to the Koopman operator. The natural assumption is that the sampling process is constructed so the transition kernel $\bP(Y \in A \mid X = x)$ is the same as that of the dynamical system of interest. For example, if the dynamical system is deterministic with $Y = f(X)$, then we would assume $Y(\omega) = f(X(\omega))$.

Different choices of sampling process $\theta$ and ergodic sampling measure $\mu$ yield different sampling patterns when it comes to $X,Y$. For example, in this framework, possibilities for the evolution of $\{X(\omega_m)\}$ include:
\begin{itemize}
	\item A sample time series from the dynamics, i.e. with $X(\omega_{m+1}) = Y(\omega_m)$;
	\item \iid{} distributed according to any probability measure;
	\item The evolution of a Monte Carlo Markov Chain sampler.
\end{itemize}

Given some set of observable functions $\psib: \Lambda \to \C^N$ and a finite number of observations \[ \{(X(\omega_m),Y(\omega_m))\}_{m=1,\ldots,M},\] we can make a least-squares approximation of the Koopman operator of $f$. If we let
\begin{align*} \hPXX &= \frac 1M \sum_{m=1}^M \psib(X(\omega_m)) \psib(X(\omega_m))^* &
	\hPXY &= \frac 1M \sum_{m=1}^M \psib(X(\omega_m)) \psib(Y(\omega_m))^* 
	\end{align*}
then our least-squares Koopman operator approximation is
\[ \hat K = \hPXX^{-1} \hPXY, \]
provided $\hPXX$ is invertible. 

If we define for $\lambda \in \C$ the snapshot characteristic matrix
	\[ \cl{\omega} = \lambda \psib(X(\omega)) \psib(X(\omega))^* - \psib(X(\omega)) \psib(Y(\omega))^* \]
we then notice that the sample characteristic matrix is just
\[ \hCl = \lambda \hPXX - \hPXY\]
and so in our set-up 
\[ \sigma(\hat K) = \left\{ \lambda \in \C : \ker (\lambda \hPXX - \hPXY) \neq \{0\} \right\} = \hat \Sigma\]
and $E = \C$.

It will be helpful to note that the snapshot characteristic matrices are rank-one (see, e.g. Section~\ref{ss:FastVariance})
\[ \cl{\omega} = \psib(X(\omega)) (\bar\lambda \psib(X(\omega)) - \psib(Y(\omega)))^*. \]
	
Taking the number of sample points $M \to \infty$, we find that 
\begin{align*} \hPXX \xrightarrow{M\to\infty} \PXX &:= \int_\Omega \psib(X) \psib(X)^*\, \d \mu &
	\hPXY \xrightarrow{M\to\infty}  \PXY &:= \int_\Omega \psib(X) \psib(Y)^*\, \d \mu
\end{align*}
This means that in the infinite data limit, our EDMD matrix converges
\[ \hat K \xrightarrow{M\to\infty} K := \PXX^{-1} \PXY, \]
and we have our true characteristic matrix 
\[ \Cl = \lambda \PXX- \PXY \]
meaning that the limiting spectrum $\sigma(K)$ is equal to $\Sigma$, matching the abstract set-up.

We note that $\sigma(K)$ is not the {\it true} spectrum of the actual, infinite-dimensional Koopman operator. Instead, it is just a projection of it to finite dictionary size. 
We do not consider the finite dictionary error in this paper, noting that unifying finite data error and finite dictionary error is known to be impossible in general for dynamical systems \cite{Colbrook24}. However, in many cases the dictionary can be chosen such that the finite data error dominates the finite dictionary size error \cite{Slipantschuk20,Wormell25,Arbabi17,Koltai23}.

\subsection{Positive definite matrices}

We will use theory of positive definite matrices heavily in this paper. 

A matrix is Hermitian if its conjugate transpose $Q^*$ is equal to $Q$. 

A Hermitian matrix $Q$ is positive definite (resp.~positive semi-definite) if $\bv^* Q \bv > 0$ (resp.~$\geq 0$) for all vectors $\bv \neq \bm{0}$. 

We will write $Q \succ B$ (resp. $Q \succeq B$) if $Q-B$ is positive definite (resp. positive semi-definite).

Furthermore, for every $N\times N$ positive definite Hermitian matrix $Q \succ 0$ we can define the following inner product on $\C^N$
\[ \langle \bv, \bw \rangle_Q := \bv^* Q \bw. \]
This induces a norm $\| \cdot \|_{Q}$. Conversely, every inner product on $\C^N$ is equal to $\langle \bv, \bw \rangle_Q$ for some $Q \succ 0$.

\subsection{Ces\`aro summation}

It will be useful to use the notion of Ces\`aro summation to deal with correlation functions. The Ces\`aro resummation of a two-sided (formal) sum is defined as 
\begin{equation} \Ces \sum_{t=-\infty}^\infty A_t := \lim_{U\to\infty} \frac{1}{U} \sum_{T=1}^U \sum_{t=-T}^{T} A_t \label{eq:Cesarodef}\end{equation}

It is precisely the kind of summation necessary to compute limiting variances of Birkhoff means of a time series. When the usual infinite sum exists, the two are equal, but the Ces\`aro sum comes into its own where there is a periodic component to $\{A_t\}_{t \in \Z}$ and the usual sum therefore does not converge.

\subsection{Assumptions on sample means}

This work is designed for operators sampled from the many random or chaotic processes with central limit theorem-type convergence of averages. 
Our results conceptually assume that the $\{\cl{\omega_m}\}$ process satisfies a central limit theorem. Let us define for any linear functional $h: \C^{N \times N} \to \C$,
\[ S_{M}(\omega,h) = \frac1M \sum_{m=1}^M h[\cl{\sigma^m(\omega)}] - h[\Cl]. \]
Then we assume:
\begin{assumption}[Central limit theorem] \label{as:birkhoff}
	For any linear $h: \C^{N \times N} \to \C$,
	\begin{itemize}
		\item The so-called Birkhoff variance of $h$
	\begin{align} \sigma^2_{h,\lambda} := \lim_{M\to\infty} M \bE[|S_M(\omega,h)|^2] = \Ces\sum_{t=-\infty}^\infty \left(\bE[h(\cl{\omega+t})^* h(\cl{\omega})] - \bE[h(\Cl)^* h(\Cl)]\right) \label{eq:Birkhoffdef} \end{align}
	is well-defined and finite, where $\Ces \sum$ indicates Ces\`aro summation.
		\item 	The $\{\cl{\omega_m}\}$ process satisfies a central limit theorem, i.e. the real and imaginary parts of $\sqrt{M} S_M(\cdot,h)$ jointly converge to a centred normal distribution in $\R^2$ whose covariance matrix has trace $\sigma^2_{h,\lambda}$.
		\end{itemize} 
\end{assumption}
This allows us to define the {\it variance operator}, which generalises the notion of Birkhoff variance to vectors, given a (semi-)norm $\| \cdot \|_Q$ with $Q \succeq 0$:
\begin{equation}
	\V_\lambda[Q] = \lim_{M\to\infty}  M \bE[\|\tCl\bv \|_Q^2] \label{eq:Vdef}.
\end{equation}

Because $\| \cdot \|_Q^2 + \alpha \| \cdot \|_B^2 = \| \cdot\|_{Q + \alpha B}^2$, $\V_\lambda$ is a linear operator. Assumption~\ref{as:birkhoff} in fact gives us a formula for $\V_\lambda$:
\begin{proposition}\label{p:VAsCorrelationSum}
	For all $\lambda \in E$, $\V_\lambda$ is a bounded linear operator on Hermitian matrices, and maps positive semi-definite matrices to positive semi-definite matrices. 
	\begin{equation*} \V_\lambda[Q] = \Ces \sum_{t = -\infty}^\infty \left(\bE[\cl{\omega}^* Q \cl{\theta^t(\omega)}] - \Cl^* Q\Cl\right)\end{equation*}
\end{proposition}
This proposition is proved in Appendix~\ref{a:VPropositions}.

 Central limit theorems are fairly well-known for a range of systems: SDEs, Markov processes with positive kernels, and many chaotic systems \cite{Liverani96, Hennion01, Khasminskii11}.

\section{Basic idea}\label{s:GeneralConcept}

We will introduce the concept behind the indicators we study. This can be pushed in two directions: firstly to establish quantitative probabilistic bounds on the location of data-driven eigenvalues (Section \ref{s:ProbabilisticBounds}) and secondly to provide a post-hoc statistical test for the location of true eigenvalues (Section~\ref{s:StatisticalTest}).

\subsection{Indicator $P(\lambda)$}\label{ss:Indicator}

Recalling our definition of the sampling error $\tCl = \hCl - \Cl$, we define the function $P: E \to [0,\infty]$ to be
\begin{align} P(\lambda) &= \sup_{Q \succ 0} P(\lambda,Q), & P(\lambda,Q) &= \inf_{\bv \neq \bm{0}} \frac{\| \Cl \bv\|_Q^2}{\V_\lambda[Q]} \label{eq:Pdef}\end{align}
where the variance operator is defined in \eqref{eq:Vdef}.

The key idea is that $P(\lambda)$ is small when $\lambda$ is likely to be in the sample spectrum $\Sigma$, and generally large when it is not. The idea of this is as follows.

Each $P(\lambda,Q)$ is a kind of Rayleigh quotient for $\hCl$, measured against its limiting sample variance through $\V_\lambda[Q]$. If $P(\lambda,Q)$ is large for some $Q$, then the sample characteristic matrix $\hCl$ is far from being non-invertible when measured in the $\|\cdot \|_Q$ norm.  In light of the CLT assumption on $\hCl$, this implies that the probability that $\lambda$ is in the sample spectrum $\hat{\Sigma}$ must be low. 

Finally, to avoid any insensitivity caused by sub-optimal choice of norm $\| \cdot \|_Q$, we maximise over all $Q$ to get $P(\lambda)$.

Under very general assumptions we have the following first result:
\begin{proposition}\label{p:PvsEigvals}
	$P(\lambda) = 0$ if and only if $\lambda \in \Sigma$.
\end{proposition}
\begin{proof}[Proof of Proposition~\ref{p:PvsEigvals}]
	If $\lambda \in \Sigma$, there exists some $\bv \neq 0$ such that $\Cl \bv = \bm{0}$, so the infimum in \eqref{eq:Pdef} is zero for all $Q \succ 0$. 
	
	On the other hand, if $\lambda \notin \Sigma$, then we can fix any $Q \succ 0$, note that $C(\lambda)$ is invertible, so $\| C(\lambda) \bv \|_Q \geq c \| \bv\|_Q$ for some $c > 0$. We then just need to assure the denominator is bounded, which is assured by Proposition~\ref{p:SAsCorrelationSum}. 
\end{proof}

\subsection{The matrix operator $\S_\lambda$}

While $P(\lambda)$ itself is an object of the infinite-data limit that is not immediately applicable, it has a lot of interesting properties. In particular, it is associated to the following linear operator $\S_\lambda: \C^{N\times N} \to \C^{N\times N}$: 
\begin{equation} \S_\lambda[Q] :=\V_\lambda[(\Cl^{-1})^* Q \Cl^{-1}] \label{eq:Sdef}\end{equation}
This operator $\S_\lambda$ has the very nice property that it preserves both Hermitian matrices, and the cone of positive semi-definite matrices. When $\lambda \notin \Sigma$, we can set $Q' = \Cl^{*} Q \Cl$ to rewrite 
\begin{equation} P(\lambda) = \sup_{Q' \succ 0} \inf_{\bv \neq \bm{0}} \frac{\bv^* Q' \bv}{\bv^* \S_\lambda[Q'] \bv}. \label{eq:PdefS}\end{equation}

The following result follows from \eqref{eq:Sdef} and Proposition~\ref{p:VAsCorrelationSum}.
\begin{proposition}\label{p:SAsCorrelationSum}
	For all $\lambda \in E \backslash \Sigma$, $\S_\lambda$ is a bounded operator on Hermitian matrices.
	
	Furthermore, it maps positive semi-definite matrices to positive semi-definite matrices.
\end{proposition}

The second statement provides some very useful structure, as positive semi-definite matrices form a cone, and so $\S_\lambda$ is a {\it cone-preserving operator} \cite{Berman94}, analogous to a positive matrix. Many analogous results follow. 

Of particular utility for computation, we will find the following as a direct application of Theorem~\ref{t:PIsSpecRad}:
\begin{proposition}\label{p:PComputation}
	For all $\lambda \in E \backslash \Sigma$, $\frac{1}{P(\lambda)}$ is the spectral radius of $\S_\lambda$, and there exists $Q^*_\lambda \succeq 0$ such that $\S_\lambda[Q^*_\lambda] = \frac{1}{P(\lambda)} Q^*_\lambda$.
\end{proposition}
This is highly important as it allows us to characterise $P(\lambda)$ easily, both theoretically and computationally.


\section{Rigorous result location data-driven eigenvalues}\label{s:ProbabilisticBounds}

In Section~\ref{ss:Indicator}, we argued that $P(\lambda)$ can be used as a way to estimate the asymptotic likelihood that $\lambda$ lies in the data-driven spectrum $\hat \Sigma$, which is a function of the random sample.

We can formalise this using Gaussian bounds for matrix norms as the engine \cite{Tropp15, Tropp11}. Unfortunately, for technical reasons this requires us to consider a modification of $P(\lambda)$ to also consider variation in the adjoint characteristic matrices $\cl{\omega}^*$. This leads to an indicator $P_\sym(\lambda)$, which tends to be close to $P(\lambda)$ but is much harder to compute.

%

\begin{align} P_\sym(\lambda) &= \sup_{Q \succ 0} P_\sym(\lambda,Q),\ & P_\sym(\lambda,Q) &:= \min \left\{  \inf_{\bv \neq \bm{0}} \frac{\bv^* Q \bv}{\bv^* \S_\lambda[Q] \bv}, \inf_{\bv \neq \bm{0}} \frac{\bv^* Q^{-1} \bv}{\bv^* \S_\lambda^*[Q^{-1}] \bv} \right\}\label{eq:Psymdef}\end{align}
where $\S_\lambda^*$ is the adjoint operator of $\S_\lambda$:
\begin{equation}
	\S_\lambda^*[Q] := \lim_{M\to\infty}  M \bE\left[ \Cl^{-1} \tCl Q (\Cl^{-1} \tCl)^* \right].\label{eq:Sstardef}
\end{equation}

The same basic properties around eigenvalues hold for $P_\sym$ as for $P$:
\begin{proposition}\label{p:PsymEigval}
	$P_\sym(\lambda) \leq P(\lambda)$, and $P_\sym(\lambda) = 0$ if and only if $\lambda \in \sigma(K)$.
\end{proposition}
\begin{proof}
	$P_\sym(\lambda) \leq P(\lambda)$ follows by comparing \eqref{eq:PdefS} and \eqref{eq:Psymdef}. For the second part, Proposition~\ref{p:PvsEigvals} gives ``if'' so that we only need to prove the ``only if'' direction. Here, we can fix any $Q \succ 0$, and need only that $\S_\lambda[Q]$ and $\S_\lambda^*[Q^{-1}]$ are well-defined, which we can get from Proposition~\ref{p:SAsCorrelationSum} and an appropriate modification thereof for $\S_\lambda^*$.
\end{proof}
In practice it seems that $P(\lambda) \gtrapprox 0.3P_\sym(\lambda)$ for the most part, but near eigenvalues this ratio appears to to $1$ (see Section~\ref{ss:Numerics1DTrue}). The major benefit of studying $P$ over $P_\sym$ is that it can be estimated via a power method (see Section~\ref{s:Algorithms}), making it much easier to compute.

The interest to us of $P_\sym(\lambda)$ lies in the following theorem, proved in Appendix~\ref{a:EigsInCurve}. 
\begin{theorem}	\label{t:EigsInCurve}
	Suppose that our samples are independent. 
	
	Let $F \subset E$ be an closed set whose boundary $\partial F$ is a rectifiable curve and, $P_\sym(\lambda) > P^*$ for all $\lambda \in \partial F$.
	
	Define $\#(\Sigma,F)$ to be the number of elements of $\Sigma$ in $F$, counting algebraic multiplicity, and similarly for $\hat \Sigma$.
	
	Then there exists a constant $C$ such that
	\begin{equation} \bP\left(\#(\Sigma,F) \neq \#(\hat \Sigma,F)\right) \leq C M \exp\left\{-  \frac{1}{1 + P^* R_{\partial F}/3}\cdot\frac{P^*}{2}M\right\} \label{eq:EigsInCurveLimit} \end{equation}
	where
	\begin{equation} R_{\partial F} = \sup_{\lambda \in \partial F} \inf_{\substack{Q\succ 0\\ P_\sym(\lambda,Q) \geq P^*}} \esssup_{\omega \in \Omega} \left\| \Cl^{-1} \cl{\omega} - I \right\|_Q < \infty \label{eq:RpartialF}\end{equation}
\end{theorem}

Broadly speaking, this is to say that the deviations of the data-driven eigenvalues from their infinite-data values are more or less controlled by level sets of $P_\sym(\lambda)$. We study this theorem with examples in Sections~\ref{s:VAR} and \ref{ss:Numerics1DTrue}.

%

If the size of $\cl{\omega}$ is relatively uniform, we can replace the $P^* R_{\partial F}/3$ part of the fraction in \eqref{eq:EigsInCurveLimit}, which is a maximum and therefore difficult to estimate, with an $L^2(\mu)$ estimate. This leads to the following rule of thumb justified in Appendix~\ref{a:RuleOfThumb}. It is most accurate when the $\cl{\omega}$ are rank-one and relatively uniform in size, such as in EDMD:
\begin{quotation}
	\it Suppose $P_\sym(\lambda) \geq P^*$ on the boundary of $F$.
	
	Then if $M \gg 2/P^*, 2\sqrt{N/P^*}$, the number of eigenvalues of $\hat \Sigma$ in $F$ is highly likely to be the same as that of $\Sigma$ in $F$.  
\end{quotation}

\section{Statistical test for true eigenvalues}\label{s:StatisticalTest}

The sampling pseudospectrum is a very helpful concept to understand spectral effects of approximation from finite data. Now, we would like to apply it in a context where we don't have access to the true process, but rather a finite amount of data. To do this, we will replace the infinite-data objects by finite-data estimators. In particular, we will set
\[ \hat P(\lambda) = \sup_{Q \succ 0} \inf_{\bv \neq \bm{0}} \frac{\| \hCl \bv\|_Q^2}{\bv^* \hat\V_\lambda[Q] \bv},\]
where $\hat\V_\lambda[Q]$ estimates $\V_\lambda[Q]$. We will make the following assumptions on $\hat\V_\lambda[Q]$:
\begin{assumption}\label{as:Vhat}
	$\hat\V_\lambda$ has the following properties:
\begin{itemize}
	\item $\hat\V_\lambda[Q]$ is an asymptotically consistent estimator of $\V_\lambda[Q]$ (see \eqref{eq:Vdef}) in that \[\lim_{M\to\infty} \hat\V_\lambda[Q] - \V_\lambda[Q] = 0 \textrm{ in probability}.\]
	\item $\hat\V_\lambda[Q]$ is linear in $Q$.
	\item $\hat\V_\lambda[Q]\succeq 0$ if $Q \succeq 0$.
\end{itemize}
\end{assumption}
The construction of good estimators satisfying Assumption~\ref{as:Vhat} is given in Section~\ref{s:Variance}.

$\hat P(\lambda)$ has some basic properties analogous to $P(\lambda)$:
\begin{proposition}\label{p:PhatComputation}
	$\hat P(\lambda) = 0$ if and only if $\lambda \in \hat\Sigma$.
	
	Furthermore, for all $\lambda \in E \backslash \hat \Sigma$, $\hat P(\lambda)$ is reciprocal of the spectral radius of $\hat\S_\lambda$:
	\[ \hat\S_\lambda[Q] := \hat \V[(\hCl^{-1})^* Q \hCl^{-1}]. \] 
\end{proposition}

Furthermore, the following theorem, proved in Appendix~\ref{a:StatTest}, tells us how $\hat P(\lambda)$ behaves as $M \to \infty$. It tells us that $\hat P(\lambda)$ is a consistent estimator of $P(\lambda)$, and provides a distribution for $M \hat P(\lambda)$ in the generic case of $P(\lambda) = 0$ (i.e. $\lambda \in \Sigma$).
\begin{theorem}\label{t:StatTest}
	The following hold:
	\begin{itemize}
			\item For all $\lambda \in E$,
						\[ \lim_{M\to\infty} \hat P(\lambda) = P(\lambda) \textrm{ almost surely.} \]
		\item\label{tp:StatTestb} For all $\lambda \in \Sigma$ with geometric multiplicity $1$, we have the tail bound
		\[ \lim_{M\to\infty} \bP\left(M \hat P(\lambda) > c\right) \leq \max\left\{\bP(\chi_1^2 > c),\bP(\tfrac{\chi_2^2}{2} > c)\right\} \]
		 for all $c > 0$, where $\chi_j^2$ is the chi-square variable with $j$ degrees of freedom.
	\end{itemize}
\end{theorem}
We conjecture that for $\dim \ker C_\lambda=d > 1$, the right-hand bound in part \ref{tp:StatTestb} will be $\max_{i\leq 2d} \bP(\tfrac{\chi_i^2}{i} > c)$, which is identical to the above bound for $c \geq 1.3$.

This result suggests that statistically test for the location of  true eigenvalues $\lambda \in \Sigma$. This is because as $M \to \infty$, we can see one of two things: if $\lambda \in \Sigma$, $M \hat P(\lambda)$ converges to a distribution upper-bounded by something known, whereas if $\lambda \notin \Sigma$, $M \hat P(\lambda)$ grows linearly almost surely (i.e., is distributed far away from the first distribution). We can therefore test the null hypothesis $\lambda \in \Sigma$ using $p$-value
\[ p = \min\left\{ 1 - F_{\chi_1}(M \hat P(\lambda)),  1 - F_{\chi_2}(2M \hat P(\lambda)) \right\}. \]
where $F_D$ is the cumulative distribution function of the distribution $\lambda$.

In dynamical systems, these kinds of statistical tests can be used to rule in or out the existence of certain kinds of patterns for the operator, for example resonances of certain frequencies or decay rates. The performance of this theorem is investigated in Sections~\ref{ss:Numerics1DData}.

We can draw a longer bow here. Because $\hat P(\lambda)$ consistently estimates $P(\lambda)$, away from true eigenvalues the ratio of $M \hat P(\lambda)$ to $M P(\lambda)$ is close to $1$. If $M \hat P(\lambda)$ is large (i.e. substantially greater than $1$) on the boundary of a set $F$,  it is therefore also likely that $M P(\lambda)$ is large on $\partial F$, and so by Theorem~\ref{t:EigsInCurve} the number of true and finite-data eigenvalues inside $F$ are likely to match. We therefore expect our estimate of the sampling pseudospectrum $\hat P(\lambda)$ to have a similar eigenvalue-counting property as the sampling pseudospectrum itself. This behaviour is investigated in Sections~\ref{ss:Numerics1DData} and~\ref{s:NumericsL63}.

\section{Computing $P(\lambda)$ and $\hat P(\lambda)$}\label{s:Algorithms}

We will see in this section that the great strength of the indicators $P(\lambda)$ and $\hat P(\lambda)$ (as compared to, say $P_\sym(\lambda)$) is that they are easy to calculate and can even be approximated with rigorous guarantees. The following theorem, proved in Appendix~\ref{a:PIsSpecRad}, explains:

\begin{theorem}\label{t:PIsSpecRad}
	Consider an optimisation problem of the form
	\begin{equation} P = \sup_{Q \succ 0} \inf_{\bv \ne 0} \frac{\| L \bv \|_Q^2}{\bv^* \V[Q] \bv}  \label{eq:GenOptimProblem}\end{equation}
	where $L$ is an invertible matrix, and $\V$ is linear and maps positive semi-definite matrices to positive semi-definite matrices.
	
	Let 
		\[ \S: Q \mapsto \V[(L^{-1})^* Q L^{-1}]\]
		be an operator on Hermitian matrices.
	Then
	\begin{enumerate}[a.]
		\item  \label{tt:PIsSpecRad-cone} $\S$ preserves the cone of positive semi-definite matrices, and therefore it has a leading non-negative eigenvalue equal to its spectral radius $\specrad \S \in [0,\infty)$. This eigenvalue must have at least one positive semi-definite eigenmatrix.
		\item \label{tt:PIsSpecRad-specrad} The spectral radius $\specrad S = 1/P$
		\item \label{tt:PIsSpecRad-bounds} Let $\sigma_Q$ be the generalised spectrum of $Q$ against $\S[Q]$, i.e. \[\sigma_Q = \{ \lambda \in \R : (\exists \bv) Q \bv = \lambda \S[Q] \bv \}.\]
		Then, for any $Q \succeq 0$ with $\S[Q] \succ 0$, we have 
		\[ P \in \left[\min \sigma_Q, \max \sigma_Q \right]. \]
		\item \label{tt:PIsSpecRad-limit} Suppose $\S$ is primitive\footnote{In all numerical investigations, $\S$ has turned out to be primitive.}, i.e. there exists $k$ such that for all $Q\succeq 0$, $\S^k \succ 0$. Then for all $Q \succeq 0$,
		\[ \min \sigma_{\S^k[Q]},\, \max \sigma_{\S^k[Q]} \xrightarrow{k\to\infty} P \]
		exponentially quickly.

	\end{enumerate}
\end{theorem}
The proof of this theorem is in Appendix~\ref{a:PIsSpecRad}. The key idea is that the set of positive semi-definite Hermitian matrices forms a cone on which $\S$ is invariant, and this means that the operator $\S$ has many properties analogous to those of positive matrices, such as a Perron--Frobenius theorem and min--max results \cite{Berman94}.

This theorem provides us a recipe for computing $P$-like indicators:
%
%
%
\begin{enumerate}
	\item Initialise some $Q_0 \succ 0$. This could be the identity matrix, or an optimal $Q$ for a similar problem (e.g. $Q^*_{\lambda + \Delta \lambda}$, as one numerically sweeps through a region of interesting $\lambda$). 
	\item Compute $Q_{t+1} = \S[Q_t]$ until the ratio of the largest and smallest generalised eigenvalue of $Q_t$ against $Q_{t+1}$ is smaller than $1+\epsilon$ for a specified tolerance $\epsilon$.
	\item Return the smallest generalised eigenvalue, a certified lower bound for $P$ to relative tolerance $\epsilon$.
\end{enumerate}
This can be done in $\mathcal{O}(\log \epsilon^{-1})$ applications of $\S$. Bear in mind that for practical purposes we do not need the relative tolerance $\epsilon$ to be very small: it could be $0.1$, for instance. In such cases, with good initialisation (e.g. $Q^*_{\lambda + \Delta \lambda}$), only a couple of iterations may be required.

\section{Constructing $\hat{\V}_\lambda[Q]$}\label{s:Variance}

Almost all the pieces are in place for our framework: there remains one quite minor detail, which is how to construct a good variance estimator $\hat V_\lambda$ that satisfies Assumption~\ref{a:VPropositions}. Recall this estimator seeks to estimate from finite data the Birkhoff variance operator 
\[ \V_\lambda[Q] = \lim_{M\to\infty} M \bE[\tCl^* Q \tCl], \]
recalling that $\tCl = \frac{1}{M} \sum_{m=1}^M \cl{\omega_m} - \Cl$.

Unless our data points are all independently sampled, our snapshot characteristic matrices $\cl{\omega_m}$ will be correlated, so estimating a variance of their sum will require taking into account these correlations. Indeed, we can expand \eqref{eq:Birkhoffdef} into matrix form to obtain
\begin{equation} \V_\lambda[Q] = \Ces\sum_{\ell=-\infty}^\infty \Gamma_\ell, \label{eq:VasCesaro}  \end{equation}
where the matrices 
\[ \Gamma_\ell = \bE[\cl{\theta^\ell(\omega)})^* Q \cl{\omega}] - \bE[\Cl^* Q \Cl] \]
obey $\Gamma_{-\ell} = \Gamma_\ell^*$.


For $|\ell| \leq M$ we can estimate the $\Gamma_\ell$ from our time series:
\begin{align*} \hat\Gamma_\ell = \frac{1}{M} \sum_{m=1}^{M-\ell} (\cl{\omega_{m+\ell}}-\hCl)^* Q (\cl{\omega_m} - \hCl)
,\, \ell\geq 0.\end{align*}
It is easy to check that this preserves positive semi-definiteness of $\hat \V_\lambda[Q]$.

We could therefore estimate $\V_\lambda$ in \eqref{eq:VasCesaro} as one of the finite sums in \eqref{eq:Cesarodef} to which the Ces\`aro limit converges. However, it will be more advantageous in terms of convergence for us to choose a window $L' + L_M$, some appropriate weights $\kappa_M(\ell)$ and set
\[ \hat \V_\lambda[Q] := \sum_{\ell=-(L'+L_M)}^{L' + L_M} \kappa_M(\ell) \hat \Gamma_\ell. \] 
A computationally efficient way to compute $\hat V_\lambda$ is given in Section~\ref{ss:FastVariance}.

The following proposition, proved in Appendix~\ref{a:VPropositions}, tells us that we have a fair amount of flexibility to ensure that $\hat \V_\lambda$ is a consistent, positive semi-definite estimator:
\begin{proposition}\label{p:VhatConstruction}
Suppose that $\kappa_M = \kappa_p \ast \kappa_w(\cdot/L_M)$, where:
\begin{itemize}
	\item $\lim_{M\to\infty} L_M = \infty$ and $\lim_{M\to\infty} \frac{L_M}{\sqrt{M}} = 0$.\footnote{This can be relaxed to $\lim_{M\to\infty}L_M / M = 0$ if the $\cl{\omega}$ are fourth-order exponential mixing, as is typical for stochastic and chaotic systems.}
	\item There exists a constant $C$ such that for all $M$ and $|\ell| \leq L_M + L'$, $\bV[\hat \Gamma_\ell] \leq C/M$.
	\item $\kappa_p: \Z \to \R$ is supported on $-L',\ldots, L'$ and $\sum_{\ell = -L'}^{L'} \kappa_p(\ell) = 1$.
	\item $\kappa_w: \R \to \R_+$ is supported on $[-1,1]$, $\kappa_w(0) = 1$, and $\kappa_w'$ is of bounded variation.
	\item The Fourier transforms of $ \kappa_p$ and $ \kappa_w$ are non-negative.
\end{itemize}
Then $\hat\V_\lambda$ satisfies Assumption~\ref{as:Vhat}.
\end{proposition}
This result tells us that we can estimate our Ces\`aro sum while maintaining positive definiteness but choosing an appropriate smooth kernel $\kappa_w$. However, we can also accelerate the convergence of our variance estimator $\hat \V_\lambda$ by cancelling out any slowly-decaying or quasi-periodic resonances in the sample data, by carefully choosing $\kappa_p$. We will discuss this in the rest of the section.

\subsection{Windowing kernel $\kappa_w$}

We begin by considering $\kappa_w$. For now, we will assume that $\kappa_p$ is trivial (so $L' = 0$). 

Here, we can only capture $\Gamma_\ell$ for $|\ell| \leq L_M$, and in practice we would like to keep $L_M$ small as possible so as to minimise the standard deviation of $\hat\V_\lambda$.


We are trying to approximate a sum over all the $\Gamma_\ell$ \eqref{eq:VasCesaro} by a weighted sum
\[  \Ces \sum_{\ell = -\infty}^\infty \Gamma_\ell \approx \sum_{\ell=-\infty}^\infty \kappa_M(\ell) \Gamma_\ell. \]

Supposing the $\Gamma_\ell$ decay quickly as $\ell \to \infty$, the Ces\`aro sum \eqref{eq:VasCesaro} reduces to a regular sum, and the largest contributions will be at $|\ell|$ small. This means we want our $\kappa_M(\ell)$ to be close to $1$ for $\ell$ small, while still satisfying the assumptions of Proposition~\ref{p:VhatConstruction}. Using $\kappa_w$ for this, we seek $\kappa_w(x) \approx 1$ for $x \ll 1$: unfortunately, the assumptions of Proposition~\ref{p:VhatConstruction} force $\kappa_w''(0) < 0$. However, the following $\kappa$ at least minimises $|\kappa''(0)|$ within the constraints on $\kappa_w$:
\[ \kappa_w(x) = \begin{cases} \frac{1}{\pi} \sin \pi|x| + \left(1-|x| \right)\cos \pi x, &|x| < 1 \\  0 &|x| \geq 1 \end{cases} \]
This function is plotted in Figure~\ref{f:bumpfunc}

\begin{figure}
	\centering
	\includegraphics{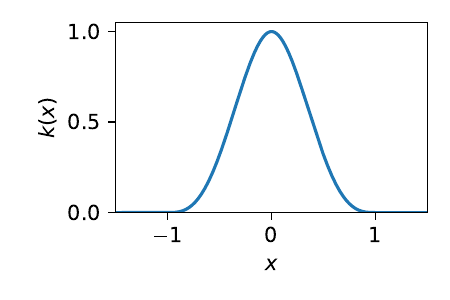}
	\caption{Plot of the optimal windowing function $\kappa_p$.}
	\label{f:bumpfunc}
\end{figure}

\subsection{Metastability kernel $\kappa_p$}
Now, suppose the $\Gamma_\ell$ decay slowly. If the sampling process is chaotic or stochastic, we usually find that our correlations have some nice summation structure, i.e. there exists an expansion \[\Gamma_\ell = \sum_{k=1}^K W_k \mu_k^{\ell} + \mathcal{O}(e^{-\ell/\tau}),\, \ell \geq 0 \]
where the $\mu_k$ lie on or inside the complex unit circle, and $e^{-1/\tau}$ lies far enough inside the unit circle that the $\mathcal{O}(e^{-\ell/\tau})$ decays at a reasonably fast exponential rate. 
In general, the $\mu_k$ are peripheral eigenvalues of a suitably formulated Koopman operator of the sampling process $\theta$ (which we may know, or can estimate from our sample). 

Let's consider the windowing function for $\mu \in \C \backslash \{1\}$:
\[ d_\mu(k) = \begin{cases}\frac{1+\mu^2}{(1-\mu)^2},& k = 0 \\ -\frac{\mu}{(1-\mu)^2},& k = \pm 1 \\0,& \textrm{ else}\end{cases} \]
It can be checked that this function has non-negative Fourier transform.
We find that for $\ell > 1$,
\[ (d_\mu \ast \Gamma_\cdot)(\ell) = \sum_{k=1}^K  \tfrac{1+ \mu\mu_k}{(1-\mu)^2} (\tfrac{\mu}{\mu_k}-1) \cdot W_k \mu_k^\ell + \mathcal{O}(r^n) \]
and by symmetry, $(\kappa_p\ast \Gamma_\cdot)(-\ell) = (\kappa_p\ast \Gamma_\cdot)(\ell)^*$. Importantly, if $\mu = \mu_k$, then we have removed one of the $\mu_k$ terms. Combining several of these kernels as
\[ \kappa_p = d_{\mu_1} \ast d_{\mu_2} \ast \cdots \ast d_{\mu_k}, \]
we find that for $|\ell| > K$,
\begin{equation} (\kappa_p \ast \Gamma_\cdot)(\ell) = \mathcal{O}(e^{-\ell/\tau}) \label{eq:windowedfastdecay} \end{equation}
and so is reasonably quickly decaying over some essential spectrum decorrelation time $\tau$. We can recast our weighted sum
\[ \hat\V_\lambda[Q] = \sum_{\ell = -\infty}^\infty \kappa_M(\ell) \Gamma_\ell = \sum_{\ell = -\infty}^\infty \kappa_w(\tfrac{\ell}{L_M}) (\kappa_p \ast \Gamma_\cdot)(\ell), \]
and so we are back to the above case.

Note that the strategy of cancelling out resonances $\mu$ is only effective when $\mu$ is far from 1 (i.e. the resonances are quasiperiodic rather than simply slowly decaying). If $1-\mu \in \C$ is small, $d_\mu$ becomes large, making the estimator unstable.

\subsection {Choice of kernel window $L_M$}

Finally, we need to settle on $L_M$. We need it to cover the large coefficients of $\kappa_p \ast \Gamma$: assuming that this is possible and $\kappa_p \ast \Gamma$ is reasonably quickly decaying, our main source of error comes from the fact that $\kappa_w \neq 1$ at the order of the decorrelation time $\tau$. In this case, for any vector $\bv$ the expected error is 
\[ \bE[ \bv^* \hat \V_\lambda[Q] \bv ] = \textrm{bias}^2 + \textrm{variance} \sim \left(\frac{\kappa_w''(0)}{2} \frac{\tau}{L_M}\right)^2 + \frac{L_M}{M}. \]
this is minimised when 
\[L_M \sim \sqrt[5]{\frac{16}{\kappa_w''(0)^2} M \tau^4} \approx \tau \sqrt[5]{\frac{16 M}{\pi^4 \tau}}.\]

In most cases, this means that $L_M$ should be a few times the decorrelation time.  

\subsection{Efficient algorithm for rank-one characteristic matrices}\label{ss:FastVariance}

Most parts of estimating $\hat\S(\lambda)$ using the methods in this and the previous section can be estimated in $\mathcal{O}(N^3 + M N^2)$ time, similar to exact DMD algorithms. The only area where this may not be the case is in in evaluating the long lag sum in $\hat V[Q]$. However, if the $\cl{\omega_m}$ are rank-one (as in EDMD), then this can be achieved in $\mathcal{O}(N^3 + M N^2 + M N L_M)$ time.

Let $\tilde L = L_M + L'$ and for $\ell \geq 0$ let 
\[ \tilde\V[Q] = \sum_{\ell = 0}^{\tilde L} \tilde \kappa(\ell) \hat \Gamma_\ell,\, \tilde \kappa(\ell) = \begin{cases} \kappa(\ell), &\ell > 0 \\ \tfrac 12 \kappa(0), &\ell = 0 \end{cases} \]
so $\hat\V[Q] = \tilde\V[Q] +\tilde\V[Q]^*$.

It can be shown that we can write 
\begin{align*} \tilde\V[Q] 
&= \frac1M \sum_{m=1}^M \left(\sum_{\ell=0}^{\min\{\tilde L, M-m\}} \tilde\kappa(\ell) \Cl[\omega_{m+\ell}]^*\right)\Cl[\omega_m] 
+ \frac1M \sum_{m=1}^{\tilde L} \left(\sum_{\ell=m}^{\tilde L} \tilde\kappa(\ell) \right) \left( \hCl^* \Cl[\omega_{M-m+1}] + \Cl[\omega_m]^* \hCl \right)
\\&\qquad - \left(\sum_{\ell=0}^{\tilde L} \frac{M+\ell}{M} \tilde \kappa(\ell) \right) \hCl^* \hCl \end{align*}

If, as in DMD applications, our snapshot characteristic matrices are rank-one, so $\Cl[\omega] = u_\omega v_\omega^*$, our first term becomes 
\[ \frac1M \sum_{m=1}^M \left(\sum_{\ell=0}^{\min\{\tilde L, M-m\}} \tilde\kappa(\ell) \left(u_{\omega_{m+\ell}}\cdot u_{\omega_{m}}\right) v_{\omega_{m+\ell}}\right)v_{\omega_m}^*
\]
which can be computed in $\mathcal{O}(M(N^2 + \tilde L N))$ operations. The second term involves adding $2L$ products of a full matrix ($\hCl$) with a rank-one matrix (a snapshot characteristic matrix), which together with computing the sums of $\tilde \kappa$ makes a $\mathcal{O}(\tilde L N^2 + \tilde L^2)$ computation. Finally, assuming that multiplying two full matrices together is done naively with $\mathcal{O}(N^3)$ operations, the last operation takes $\mathcal{O}(\tilde L + N^3)$ operations. Since $\tilde L = \mathcal{O}(L_M) = o(\mathcal{M})$, we can therefore conclude that computing $\hat V[Q]$ requires $\mathcal{O}(N^3 + MN^2 + M L_M N)$ operations.

Consequently, evaluating $\hat\S[Q]$ requires the same order of operations, and computing $\hat P(\lambda)$ to relative tolerance $\epsilon$ using the algorithm in Section~\ref{s:Algorithms} therefore requires $\mathcal{O}((N^2 + MN + M L_M )N \log \epsilon^{-1})$ operations. Given that it is reasonable to set $\epsilon$ to be quite large (e.g. $0.1$), and $L_M$ is unlikely to be much larger than $N$, this is ultimately comparable to an exact DMD algorithm.

\section{Example 1: multivariate AR process}\label{s:VAR}

We now illustrate our methods with some examples, beginning with a setting where we can compute $P(\lambda)$ explicitly. We consider an autoregressive Gaussian process given by 
\[ Y = A X + \xi \]
with $X, Y \in \R^N$, $X$ sampled again \iid $\sim \mathcal{N}(0,\Sigma_X)$ with $\Sigma_X$ invertible, and {\iid} noise $\xi \sim \mathcal{N}(0,\Sigma_\xi)$.

Our observation function will be simply $\psib(x) = x$. In this case, our snapshot characteristic matrices are
\[ \cl{\omega}= X (\lambda X^* - Y^*) = X (X^* (\lambda I - A^*) - \xi^*), \] 
and so
\[ \Cl = \Sigma_X R^{-1},\, R := (\lambda I - A^*)^{-1}. \]

We can show that 
\[ \V_\lambda[Q] = \bE[\cl{\omega}^* Q \cl{\omega}] - \Cl^* Q \Cl = \tr (\Sigma_X Q) \left( R^{-1} \Sigma_X R^{-*} + \Sigma_\xi\right) + R^{-*} \Sigma_X Q^T \Sigma_X R^{-1} \]
where $Q^T$ is the real transpose of $Q$. 
Noting that the spectral radius of $\S_\lambda: Q \mapsto \V_\lambda[\Cl^{-*}Q\Cl^{-1}]$ is the same as the spectral radius of $\S'_\lambda: Q \mapsto \Cl^{-*} \V_\lambda[Q] \Cl^{-1}$, we have that 
\[ \Sigma_X \S'_\lambda[Q] = \tr(\Sigma_X Q) \left(I + R^* \Sigma_\xi R \Sigma_X^{-1} \right) + \Sigma_X Q^T \]

By taking the trace of both sides, and noting that since $\Sigma_X$ is real symmetric so $\tr \Sigma_X Q^T = \tr Q \Sigma_X$, and using the cyclic property of the trace, we find
\[ \tr \Sigma_X \S'_\lambda[Q] = \tr(\Sigma_X Q) \left(N + 1 + \tr \Sigma_\xi R \Sigma_X^{-1} R^* \right) \]
so, choosing $Q$ to be a positive semi-definite leading eigenfunction (therefore with $\tr \Sigma_X Q > 0$), implies that
\begin{equation}P(\lambda)^{-1} = N + 1 + \tr \Sigma_\xi R \Sigma_X^{-1} R^*.\label{eq:ARP}\end{equation}

This closed form suggests that there are two effects at play here, which can be understood by thinking about $1/P(\lambda)$ as the order of snapshots required for the area around $\lambda$ to be clear of finite-data spectrum. 

Firstly, there is a base rate of $N+1$, which we can imagine as asking for enough snapshots to get into the underfitting regime and then to beat random matrix-like errors.

On top of this, noting that $ R\Sigma_X^{-1} R^*$ is the inverse of the covariance matrix of $(\lambda I - A^*) X$, we have that the necessary number of snapshots depends on how easily (and in how many different directions) variability in $\xi$ can cancel out variability in $(\lambda I - A^*)$ and thus make the fitted matrix non-invertible. Here, there is a connection to the standard notion of pseudospectrum, as
\[\tr \Sigma_\xi R \Sigma_X^{-1} R^* = \| \Sigma_\xi^{1/2} R \Sigma_X^{-1/2} \|_{\rm Frob}^2. \] 
Recall that $R = (\lambda I - A^*)^{-1} = (\lambda I - K)^{-1}$ is the standard resolvent: we might study $\| R \|_{\rm Frob}^{-1}$ as the classical pseudospectrum.

\begin{figure}[t]
	\centering
	\includegraphics[scale=\imgscale]{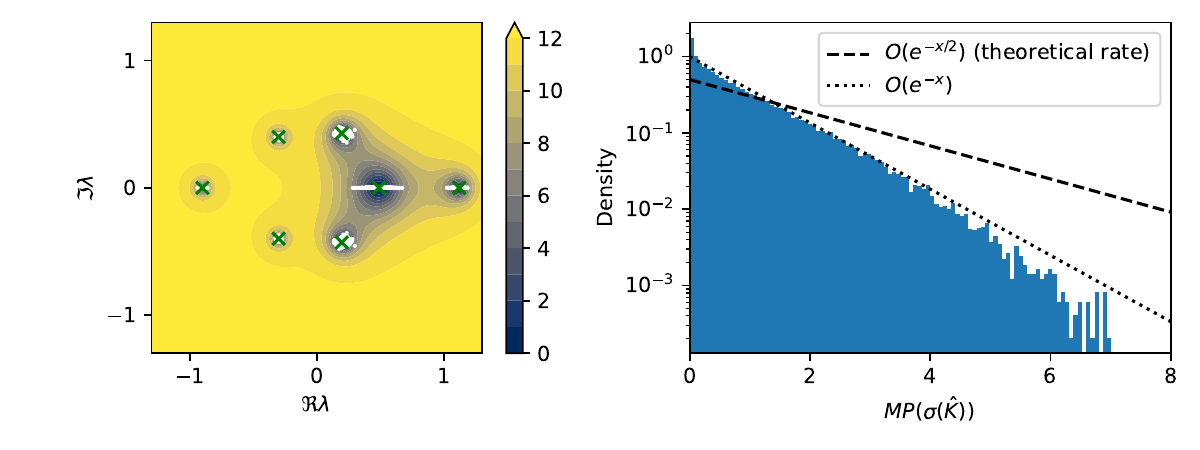}
	\caption{Left: sampling pseudospectrum for system in Section~\ref{s:VAR} with true eigenvalues (green crosses), plotted against data-driven eigenvalues (white dots) for 1,000 samples of $\hat K$ with $M = 100$. Right: histogram of $M \hat P(\lambda)$ for $\lambda \in \hat \Sigma = \sigma(\hat K)$ for 10,000 samples of $\hat K$, with exponential fits.}
	\label{f:VAR}
\end{figure}

While Theorem~\ref{t:EigsInCurve} does not apply here, as our samples are unbounded, we can attempt to test $P(\lambda)$ for $\lambda \in \sigma(\hat K)$ anyway. In Figure~\ref{f:VAR}, we considered the matrix 
\[ A = \begin{bmatrix} -0.9 & 0 & 0 & 0 & 0 & 0 & 0 \\
	0.1 & -0.3 & -0.4 & 0 & 0 & 0 & 0.1 \\
	0.1 & 0.4 & -0.3 & 0.1 & 0 & 0 & 0 \\
	0.1 & 0 & 0 & 0.5 & 0 & 0.1 & 0 \\
	0 & 0 & 0 & 1.0 & 0.5 & 0 & 0.1 \\
	0.1 & 0 & 0 & 0 & 0.9 & 0.5 & 0.1 \\
	0 & 0 & 0 & 0.1 & 0.1 & 0.9 & 0.5 \\
\end{bmatrix}
\]
with $\Sigma_X = I$, $\Sigma_\xi = 0.1 I$. Setting $M = 100$, we computed $\hat K$ 10,000 times and plotted a histogram of the values of $M P(\lambda)$ on $\sigma(\hat K)$. Theorem~\ref{t:EigsInCurve} suggests that the large tails of this value should be exponentially distributed, and this appears to be the case. However, the rate of decay appears to be better than the theorem produces: Theorem~\ref{t:EigsInCurve} would suggest that $\bP(M P(\lambda) > c) \simeq e^{-c/2}$, but in fact we see probabilities more like $e^{-c}$, a behaviour which appears more broadly in the examples to follow (see, for example, Figure~\ref{f:1DStatTest}).

However, it is worth noting that $e^{-MP(\lambda)}$ tails are certainly not universal, and $e^{-MP(\lambda)/2}$ is the tightest general bound. 
Indeed, if we choose a one-dimensional example of $X_m$ {\iid} uniformly distributed on $\{-1,1\}$, and $Y_m$ {\iid} normals, we find that
\[ K = \frac{\sum_{m=1}^M X_m Y_m}{\sum_{m=1}^M X_m^2} = \frac1M \sum_{m=1}^M X_m Y_m \sim \mathcal{N}\left(0,\frac{1}{M}\right). \]
and $P(\lambda) = |\lambda|^2 \sim \frac{1}{M} \chi_1^2$ whose tails decay as $e^{-MP(\lambda)/2}$.
We will also see $e^{-MP(\lambda)/2}$ in the estimator $\hat P(\lambda)$ in Section~\ref{ss:Numerics1DData}: it tends to be associated with real eigenvalues where $e^{-MP(\lambda)}$ is associated with complex eigenvalues (see also the end of the proof of Theorem~\ref{t:StatTest}).

	\section{Example 2: 1D expanding maps}\label{s:Numerics1D}
		\begin{figure}
		\centering
		\includegraphics[scale=\imgscale]{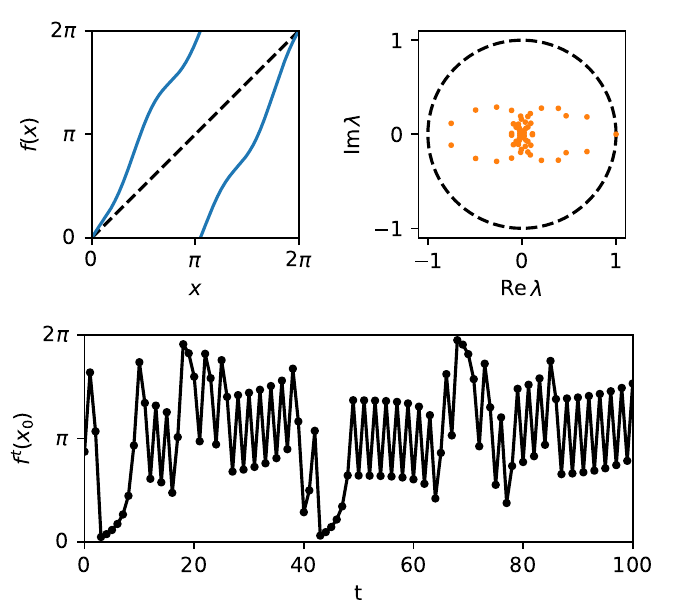}
		\caption[One-dimensional map]{Top left: graph of the map \eqref{eq:1DMap}. Top right: its Ruelle--Pollicott resonances. Bottom: a representative time series $x_{t+1} = f(x_t)$.}
		\label{f:1Dmap}
	\end{figure}
	
	We now turn to some numerical examples. We will begin by testing our theorems in a setting where we can calculate the true expectations well \cite{Wormell19, Wormell25}, namely in EDMD for an expanding map of the periodic interval. Here, and in the following sections, we will use the setup in Section~\ref{ss:KoopmanContext}.
	
	We consider the following chaotic system, an expanding map on the periodic interval $[0,2\pi)$ given by
	\begin{equation} f(x) = 2x + 2\pi \left(-0.03 + 0.04\sin x + 0.03\cos 3x - 0.03\sin 3x\right) \mod 2\pi. \label{eq:1DMap}\end{equation}
	Dynamical information about this map is given in Figure~\ref{f:1Dmap}.
	
	We will choose our observables as a trigonometric (i.e. Fourier) basis. For even $N$, this means
	\[ \psib(x) = \{1, \cos x, \sin x, \cos 2x, \sin 2x, \ldots, \sin \tfrac{(N-2)x}{2}, \cos \tfrac {Nx}2 \}. \]

	Our sample will be taken {\iid} as $X_m \sim U(0,1), m = 1,\ldots, M$, with $Y_m = f(X_m)$.
	
	This means that when we come to compute the $M\to\infty$ limit objects like $K$ and $P(\lambda)$, we can evaluate expectations as 
	\[ \bE[g(X,Y)] = \int_0^1 g(x,f(x))\,\d x. \]
	
	\subsection{``True'' indicators vs. eigenvalues}\label{ss:Numerics1DTrue}
For $N \in \{10,20,50\}$, we computed the $N\times N$ matrix $K$ and its eigenvalues, as well as the statistics $P(\lambda)$ and $P_\sym(\lambda)$ using the algorithms in Section~\ref{s:Algorithms}. Because we are in the \iid-sampled case, we have a simple formula for $\V_\lambda[Q]$ as $\bE[\cl{\omega}^* Q \cl{\omega}] - \bE[\Cl^* W \Cl]$. 

The outputs are shown in Figure~\ref{f:1DPplots}. The figures show that $P(\lambda)$ and $P_\sym(\lambda)$ are qualitatively and quantitatively similar, with the ratio $P_\sym(\lambda) / P(\lambda)$ consistently in the range $0.3$--$1$, and approaching $1$ close to simple eigenvalues.

It also appears that for fixed $\lambda$, $P(\lambda)$ decreases as $N$ increases, implying a larger data size required to ``resolve'' the spectrum. This effect is stronger as $|\lambda|$ approaches the centre of the unit circle. This suggests that in absence of regularisation, adding more observables to the dictionary can cause the Koopman spectrum to quickly become unstable: this can be understood as overfitting, but with different sensitivity in different parts of the spectrum.

	\begin{figure}[htbp]
	\centering
	\includegraphics[scale=\imgscale]{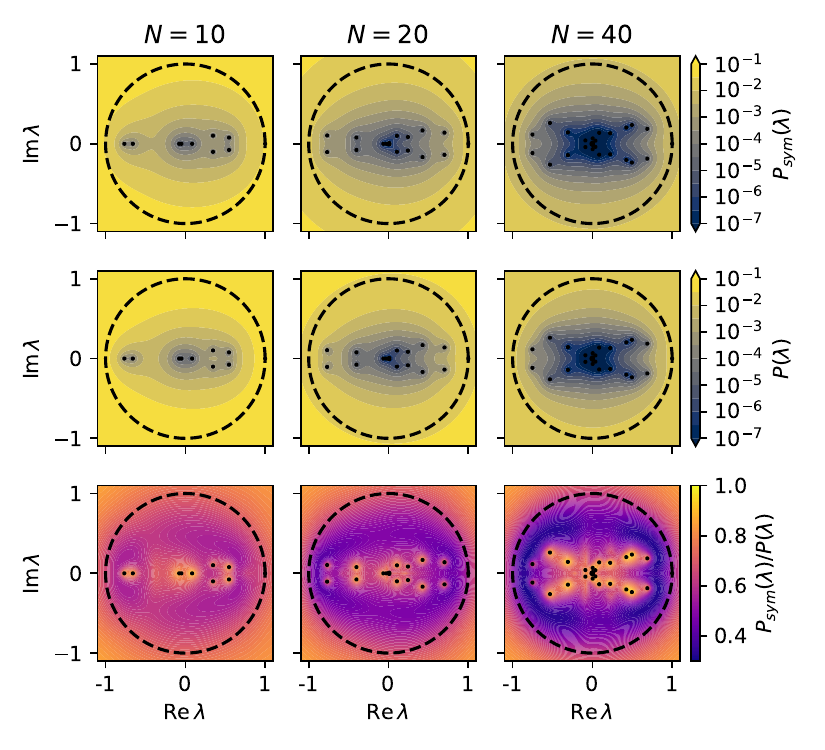}
	\caption[Comparison of $P(\lambda)$ and $P_\sym(\lambda)$]{For different values of $N$, a comparison of $P(\lambda)$ and $P_\sym(\lambda)$ for the system in Section~\ref{s:Numerics1D}. Top row: filled contour plot of indicator $P_\sym(\lambda)$; middle row: filled contour plot of indicator $P(\lambda)$; bottom row: heatmap of their ratio. In black, the eigenvalues of the corresponding infinite-data Koopman matrices.}
	\label{f:1DPplots}
\end{figure}

Against these, we simulated $K^M$ for various values of $M$ and $N$, and plotted their spectra against $P_\sym(\lambda)$ and $P(\lambda)$ (see Figure~\ref{f:1DPvsdataeigs}). Here we see that, despite highly variable distributions of eigenvalues for different $M$ and $N$, $P_\sym(\lambda)$ is a fairly efficient indicator of where eigenvalues are likely to be found: recall from Theorem~\ref{t:StatTest} that the number of data-driven eigenvalues $\lambda_{K^M}$ with $P_\sym(\lambda_{K^M})>\theta$ is of the order of $Ne^{-\theta/2}$.

Being more precise, we see that $P_\sym(\lambda)$ is a highly efficient measure of the spectrum when the error in the data-driven eigenvalues is distributed along one particular direction (for example, real eigenvalues of $K^M$, which we can see in most plots poking out of the cloud). On the other hand, it tends to somewhat overestimate the distribution of more isotropically distributed eigenvalues. For example, of the $100N = 2000$ eigenvalues in Figure~\ref{f:1DPvsdataeigs}, none have $MP_\sym(\lambda) > 2$. However, if we consult Theorem~\ref{t:EigsInCurve} bottom left and conservatively ignore any subexponential terms, we would expect around $100 \times e^{-1} \approx 36$ of them. Nevertheless, if they overestimate the true exponential rates of eigenvalue deviation by a factor of 2--4, particularly for perturbations away from the real line, $P_\sym(\lambda)$ and $P(\lambda)$ qualitatively bound the location of the data-driven eigenvalues quite accurately.
	
\begin{figure}[htbp]
	\centering
	\includegraphics[scale=\imgscale]{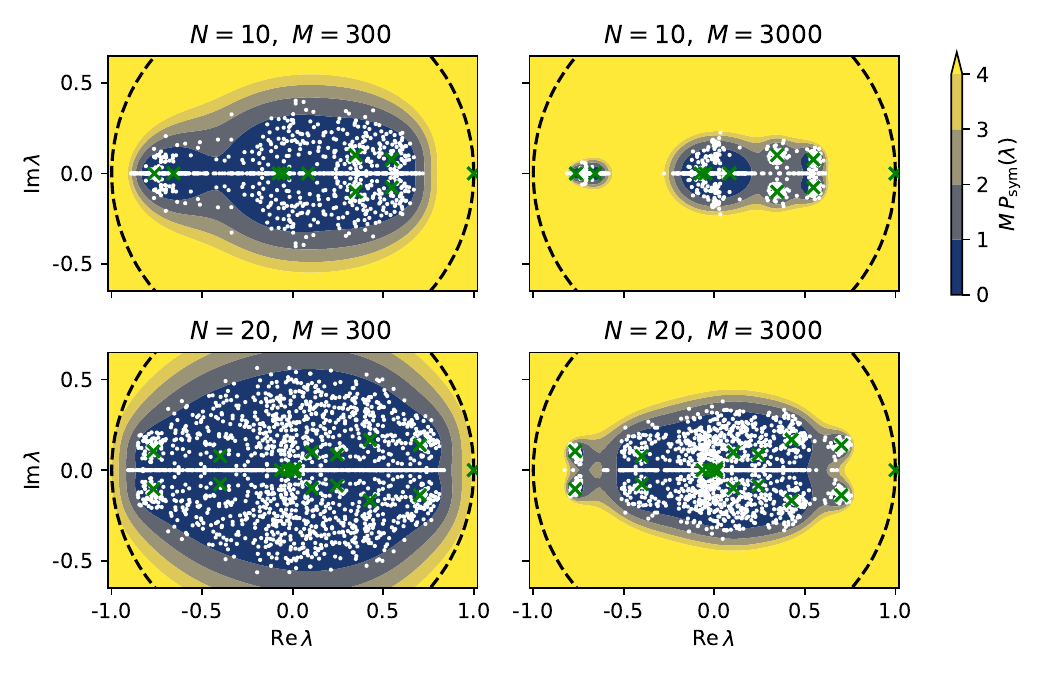}
	\caption[Plot of $P_\sym(\lambda)$ vs data-driven eigenvalues]{In white, eigenvalues of 100 Koopman matrices for the system in Section~\ref{s:Numerics1D} sampled with varying $M$. In red, eigenvalues of the true ($M\to\infty$) Koopman matrix. Background: filled contour plot $M P_\sym(\lambda)$.}
	\label{f:1DPvsdataeigs}
\end{figure}

\subsection{Data-driven $P^M$ vs. eigenvalues}\label{ss:Numerics1DData}

As we know the true eigenvalues here, we can check the results of Theorem~\ref{t:StatTest}. The first part of this theorem suggests that $\hat P(\lambda)$ is a good approximation of $P(\lambda)$, and thus of where eigenvalues are separable.

In Figure~\ref{f:1DhatPMversions}, a contour plot of $M \hat P(\lambda)$ at low levels is plotted against $M P(\lambda)$ for four different realisations. In all cases $M \hat P(\lambda)$ broadly captures the shape of $M P(\lambda)$, and correctly clusters of eigenvalues into different basins, with the likelihood depending on the depth of the basin: the cluster of two eigenvalues to the left appears in all the realisations; of the remaining ones a weak separation of the right two and, separately, the middle two, occurs in three out of the four realisations. For a comparison of the eigenvalue distributions across multiple basins, see the upper-right of Figure~\ref{f:1DPvsdataeigs}.

\begin{figure}[htbp]
	\centering
	\includegraphics[scale=\imgscale]{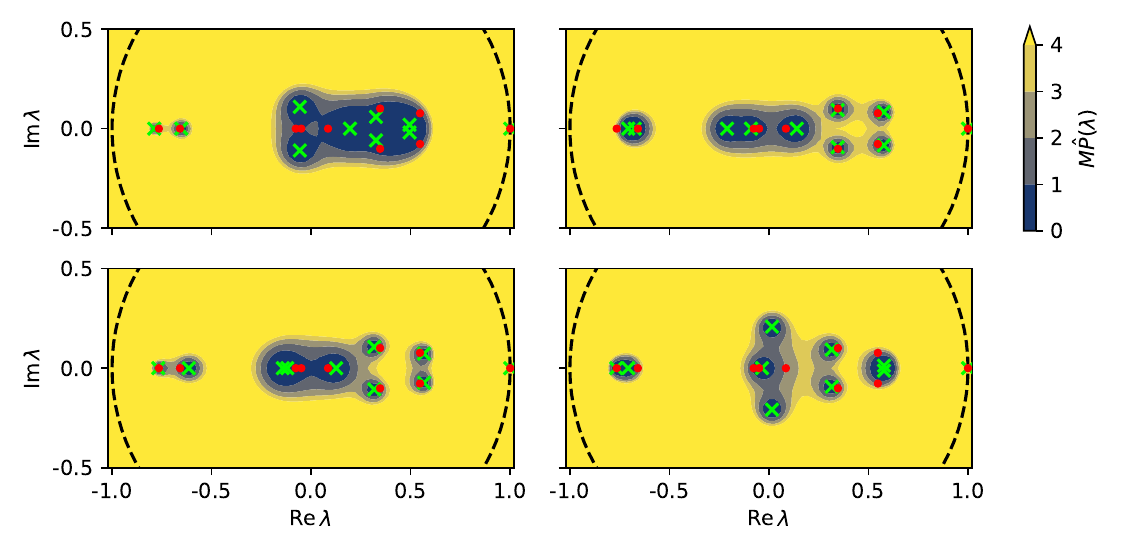}
	\caption[Plot of $M \hat P(\lambda)$ for different realisations]{For four different realisations, eigenvalues of $K^M$ (green crosses) and contour plot of $M \hat P(\lambda)$ (filled area). For comparison, eigenvalues of the true ($M\to\infty$) Koopman matrix (red dots) and corresponding contour levels of $MP(\lambda)$ (red dotted lines). The system in Section~\ref{s:Numerics1D} is used with $N=10$, $M=3000$.}
	\label{f:1DhatPMversions}
\end{figure}

The second part of Theorem~\ref{t:StatTest} suggests that true simple eigenvalues $\lambda$ of the Koopman operator tend to have low values of the data-driven indicator $\hat P(\lambda)$, and in fact the distribution of these values can be bounded. This was done in Figure~\ref{f:1DStatTest}, and it indeed appears to be the case.
 
\begin{figure}[htbp]
	\centering
	\includegraphics[scale=\imgscale]{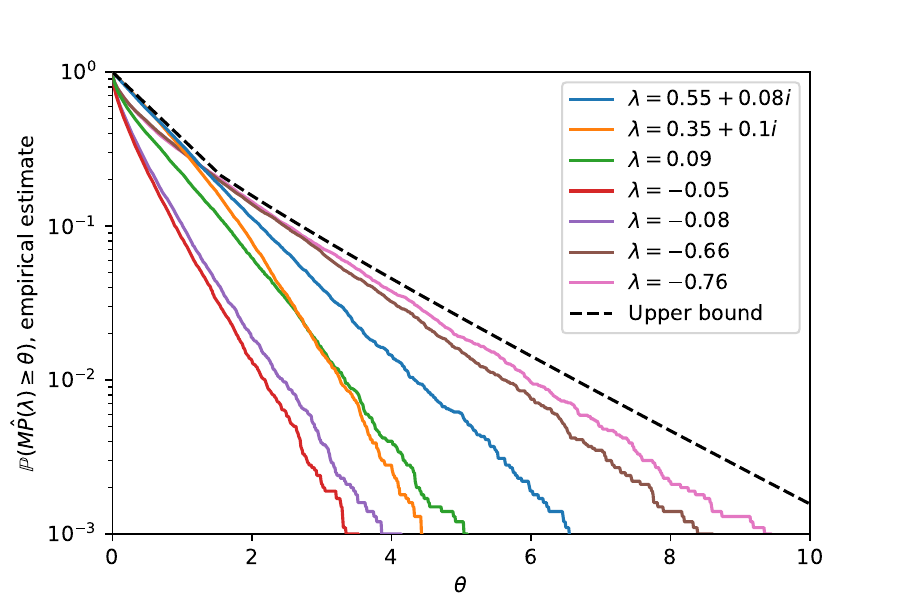}
	\caption[Distribution of $M \hat P(\lambda)$ for $\lambda \in \sigma(K)$]{Empirical estimate of $\bP(\hat M P(\lambda) \geq \theta)$ for different eigenvalues of the true Koopman operator with $N=10$, $M=3000$ for the system in Section~\ref{s:Numerics1D}. $10,000$ realisations were used.}
	\label{f:1DStatTest}
\end{figure}

%

\section{Example 3: Lorenz-63 system}\label{s:NumericsL63}

Estimation of $\hat P(\lambda)$ can be done very effectively in more complex systems as well. In this case we consider the well-known Lorenz-63 system of 3 ODEs \cite{Lorenz63}:
\begin{align*}
	\dot x &= 10(y-x)\\
	\dot y &= x(28-z) -y \\
	\dot z &= xy - \tfrac 83 z
\end{align*}
This system generates a strange attractor consisting of two ``loops'' split by a saddle point that they both flow back into (see Figure~\ref{f:L63Eigenfunctions1}). 

To study the Ruelle--Pollicott resonances, we sampled a time series of the system at $M=10,000$ instances, separated by timestep $\delta t = 0.2$. We used an EDMD observable dictionary of polynomial delay variables:
\[ \{1\} \cup \left\{ (x^i y^j z^k) \circ f^{-d} : i,j,k \in \mathbb{N},\, 1 \leq i+j+k \leq 3,\, d \in \{ 0,\ldots 9\}\right\}, \]
where $f$ is the time-$\delta t$ map of the Lorenz-63 system. This gives a total of $N = 151$ observables. As is typical, the EDMD eigenvalues have a central bulk comprising most of the eigenvalues, with a few isolated eigenvalues on the edge (see Figure~\ref{f:L63PhatM}).

We computed $\hat P(\lambda)$ for a range of $\lambda$, using the sampling kernel suggested in Section~\ref{s:Variance}: the parameters we used were $L_M = 20$, and our sampling eigenvalues were the top $9$ non-unit eigenvalues of our EDMD Koopman approximation.

\begin{figure}
	\centering
	\includegraphics[scale=\imgscale]{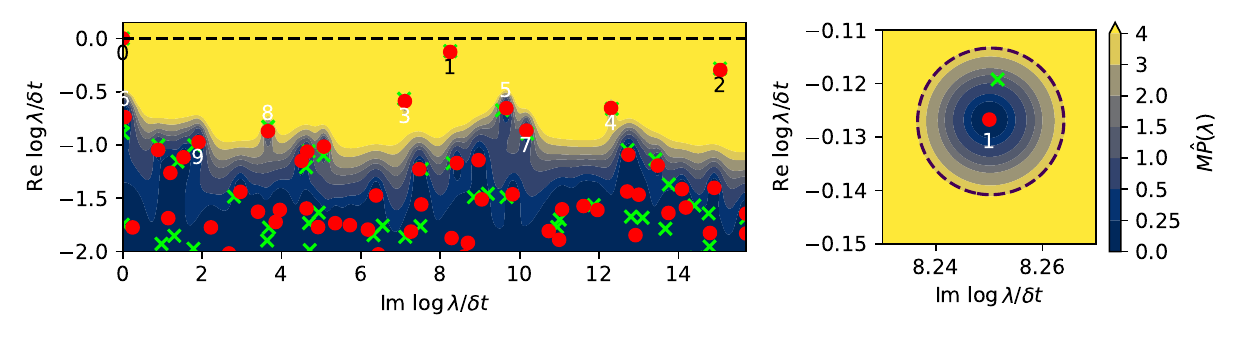}
	\caption{Left: $M \hat P$ landscape from Lorenz-63 system computed in Section~\ref{s:NumericsL63}; right: detail around eigenvalue \#1 plot. Both eigenvalues of a finite-data approximation ($M = 10,000$, red dots) and of the infinite-data limit (estimated by $M = 10^7$, green crosses) are shown. In right, the boundary of a 95\% confidence region for eigenvalue \#1 described in the text is plotted in purple.}
	\label{f:L63PhatM}
\end{figure}

The computed $M \hat P(\lambda)$ are plotted in Figure~\ref{f:L63PhatM}. We can make some observations:
\begin{itemize}
	\item For low $\Re h^{-1} \log \lambda$, i.e. small $|\lambda$|, there is a large, continuous continuous area of small $\hat P(\lambda) \ll 1$ containing most of the operator eigenvalues, specifically the unlabelled ones and those labelled \#5--9. The infinite-data eigenvalues of the EDMD matrix are found throughout this region, somewhat independently of the finite data points. Hence, these data-driven eigenvalues and corresponding eigenvectors are unlikely to reveal much of interest.
	\item As $|\lambda| \to 1$, $M \hat P(\lambda)$ grows substantially larger than $1$, indicating the unlikeliness of true eigenvalues being in this area. In this area, there are ``islands'' of small $M \hat P(\lambda)$, each containing a single finite-data eigenvalue, which we can see is matched to a single true eigenvalue.
\end{itemize}

Because $\hat P$ estimates $P$ (Theorem~\ref{t:StatTest}), and because $M P$ approximately bounds the negative log-likelihood of finite-data eigenvalues running away from their limits (Theorem~\ref{t:EigsInCurve}), we can therefore conclude that the eigenvalues \#0--4 are likely to be reflective of the infinite-data eigenvalues, and we should include them as part of any further study or dimension reduction of the system using Koopman eigenfunctions. On the other hand, we should be more suspicious of eigenvalues \#5--9. This is indeed borne out by comparing to the infinite-data eigenvalues (green crosses in Figure~\ref{f:L63PhatM}), notwithstanding that eigenvalue \#7 in particular matches a true eigenvalue quite well. 


	\begin{figure}[htb]
	\centering
	\includegraphics[scale=\imgscale]{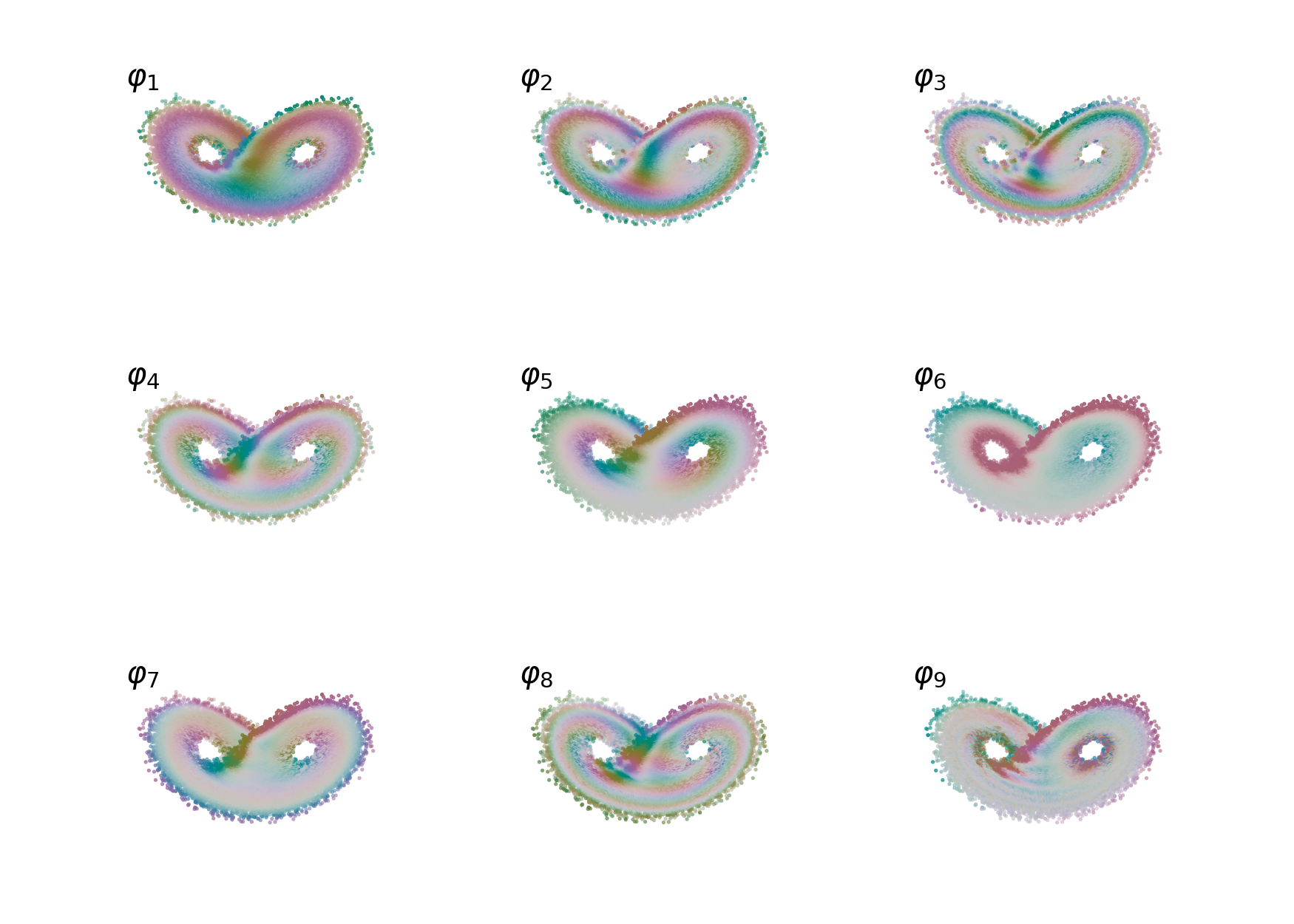}
	\caption{Leading EDMD eigenfunctions of the Lorenz-63 system computed in Section~\ref{s:NumericsL63}. Hue corresponds to complex argument, and saturation to complex modulus.}
	\label{f:L63Eigenfunctions1}
\end{figure}

\begin{figure}[htb]
	\centering
	\includegraphics[scale=\imgscale]{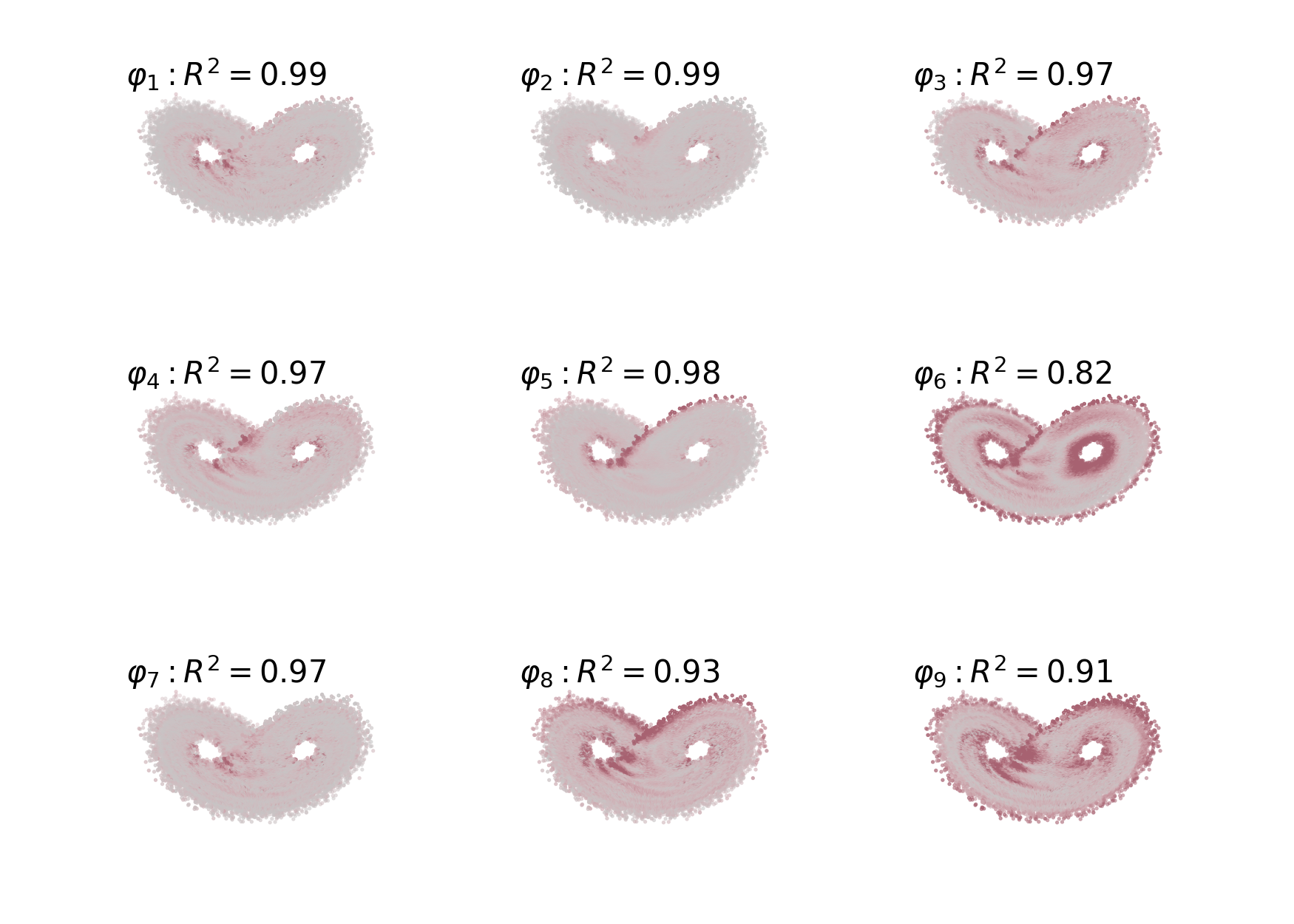}
	\caption{Least-squares residuals between large-data ($M=10^7$) and fixed-data ($M = 10,000$) eigenfunctions of the Lorenz-63 system computed in Section~\ref{s:NumericsL63}. Eigenvalues of $M=10,000$ are compared to the closest eigenvalue of the $M=10^7$ setting, which may be non-unique.}
	\label{f:L63Eigenfunctions2}
\end{figure}

The data-driven eigenvectors are plotted in Figure~\ref{f:L63Eigenfunctions1}. It is interesting to remark on the eigenfunction estimates obtained using EDMD. Because the Lorenz system exhibits highly regular periodic behaviour, the argument of the leading non-trivial eigenfunction (\#1) approximates an isochron function going around the loops. Because Theorem~\ref{t:StatTest} gives that $\lim_{M\to\infty} \bP(M\hat P \geq 3.84) \leq 0.05$, we can use the $P_M$ landscape to plot an approximate 95\% confidence region for the leading non-trivial eigenvalue $\lambda_1$ (Figure~\ref{f:L63PhatM}, right). In particular, we can conclude that $| \lambda_1 - (-0.127 + 8.25i) | \leq 0.015$. We can therefore say from the data that, finite dictionary error aside, the Lorenz attractor exhibits periodic behaviour with period $2\pi/\Im \lambda \in [0.760,0.762]$, which breaks down on timescale $T_{1/e} = 1/\Re \lambda \in [7.0,9.0]$. This is confirmed by a much larger simulation (green cross in Figure~\ref{f:L63PhatM}, right).

Beyond this, many of the other eigenvalues also contain a similar isochron-like component. In particular, eigenfunctions \#2 and \#3 are {\it approximately} second and third powers of the complex conjugate of the first eigenfunction (i.e. the eigenfunction of the complex conjugate of eigenvalue \# 1), and when reflection of the spectrum about $\Im z = \pi \delta t$ due to the sampling rate is taken into account, these eigenvalues are respectively around twice and three times the conjugate of eigenvalue \#1. This kind of approximate multiplicative structure is interesting as {\it bounded} Koopman eigenfunctions are known to have a multiplicative structure \cite{Koopman31, Mezic05}: however, the generalised eigenfunctions associated with Ruelle--Pollicott resonances of true Koopman operator \cite{Melbourne18}, which EDMD is expected to capture \cite{Wormell25}, are expected to be hyperdistributions with no multiplicative structure \cite{Baladi18}.

While we haven't explicitly proven a relationship between $M \hat P$ and eigenvector error, one seems evident. The discrepancy between the finite-data and the infinite-data eigenfunctions are plotted in Figure~\ref{f:L63Eigenfunctions2}, as well as a measure of their correlation ($R^2$). There is a clear relationship between the saddlepoint value of $M \hat P$ separating the eigenvalue from other eigenvalues and the correlation between the true and finite-data eigenvector. In particular, eigenvectors whose eigenvalues are better separated in $M \hat P$ reflect the true eigenvectors much more clearly (as demonstrated by a high $R^2$), in approximate relation to the level of the saddle. In particular, eigenfunction \#7, separated by a saddle of height approximately $M\hat P \approx 0.4$, is better resolved than the eigenfunctions \#5, 6, 8, 9. These exhibit behaviours not expected in a Koopman eigenfunction: their magnitude oscillates along the direction of the flow, and contain the same dynamical information up to linear relationships, with \#5 and \#6 appearing near-duplicates of each other.

\section{Example 4: Rayleigh--B\'enard convection}\label{s:NumericsRB}
	
We now test the method on a more complex problem, namely Rayleigh-B\'enard convection \cite{Berge84}, a classic fluids problem where an incompressible fluid is heated from the bottom in a tank. We will model a setting where this system is poorly observed (i.e. from a restricted dictionary) over a relatively long time.

In particular, we considered a 2D tank of size $1\times 1.5$ (i.e. taller than it was wide). Temperatures at top and bottom were fixed with a temperature difference of $1$. The fluid parameters were Prandtl number $0.71$, Rayleigh number $4 \times 10^7$, and Gebhart number $1.0$.

This system at equilibrium has two rotating cells of fluid stacked vertically on top of each other: in this simulation, the top cell rotated clockwise and the bottom cell rotated anticlockwise.

We simulated this system by taking $M=6,\!000$ samples at a time step of $\Delta t = 1.1$, using the Julia package {\tt IncompressibleNavierStokes.jl}. Fluid systems are sometimes observed using rods inserted into the simulation: to simulate this, our observable dictionary consisted of the temperature field at heights $0.075, 0.5, 0.925$ at 4 evenly spaced horizontal points each. We augmented this by including delay variables up to 20 steps and including a constant function, for a total of $N = 3\times 4 \times 20 + 1 = 161$ observables.

From this data, we computed the EDMD eigenvalues and sampling pseudospectrum. To compute $\hat V$ we used the method in Section~\ref{s:Variance}, using the top $4$ nontrivial EDMD eigenvalues and a general lag window of $L=200$. These are plotted in Figure~\ref{f:RBPhatM}. As usual, we see a cloud of eigenvalues in a single large cluster, with some isolated eigenvalues resolved. Apart from the trivial eigenvalue \#0, we have distinct eigenvalues \#1, \#2, \#3 and \#5 separated from the main cluster by saddles of $\hat P(\lambda) \geq 1$. From Figure~\ref{f:RBmodes} (which is generated from information we would not be able to observe in our experiment), the first three appear to correspond to Fourier modes of the snaking of the jet. The fifth mode appears to correspond to the production of blobs of hot fluid from the corner cells, which rotate around the centre of the bottom cell.

		\begin{figure}[htb]
	\centering
	\includegraphics[scale=\imgscale]{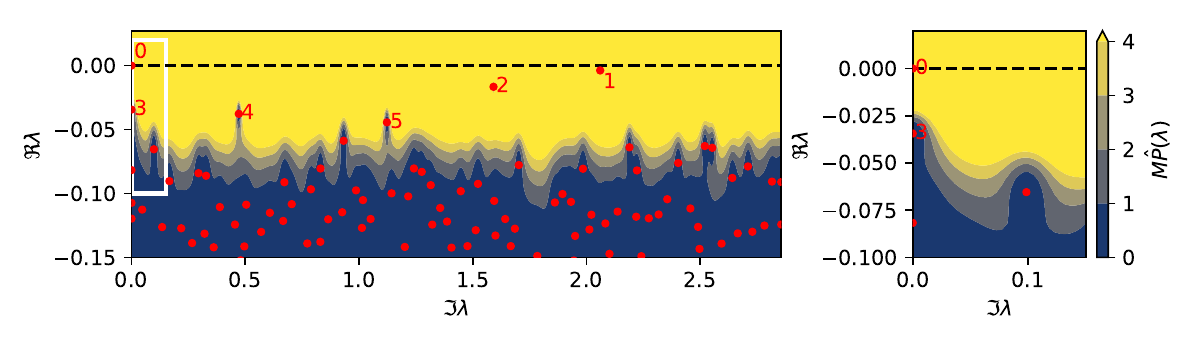}
	\caption{Eigenvalues and sampling pseudospectrum of the Rayleigh--B\'enard system computed in Section~\ref{s:NumericsRB}, plotted under transformation to continuous time. Right: inset of spectrum close to $0$}
	\label{f:RBPhatM}
\end{figure}

		\begin{figure}[h!]
	\centering
	\includegraphics[scale=\imgscale]{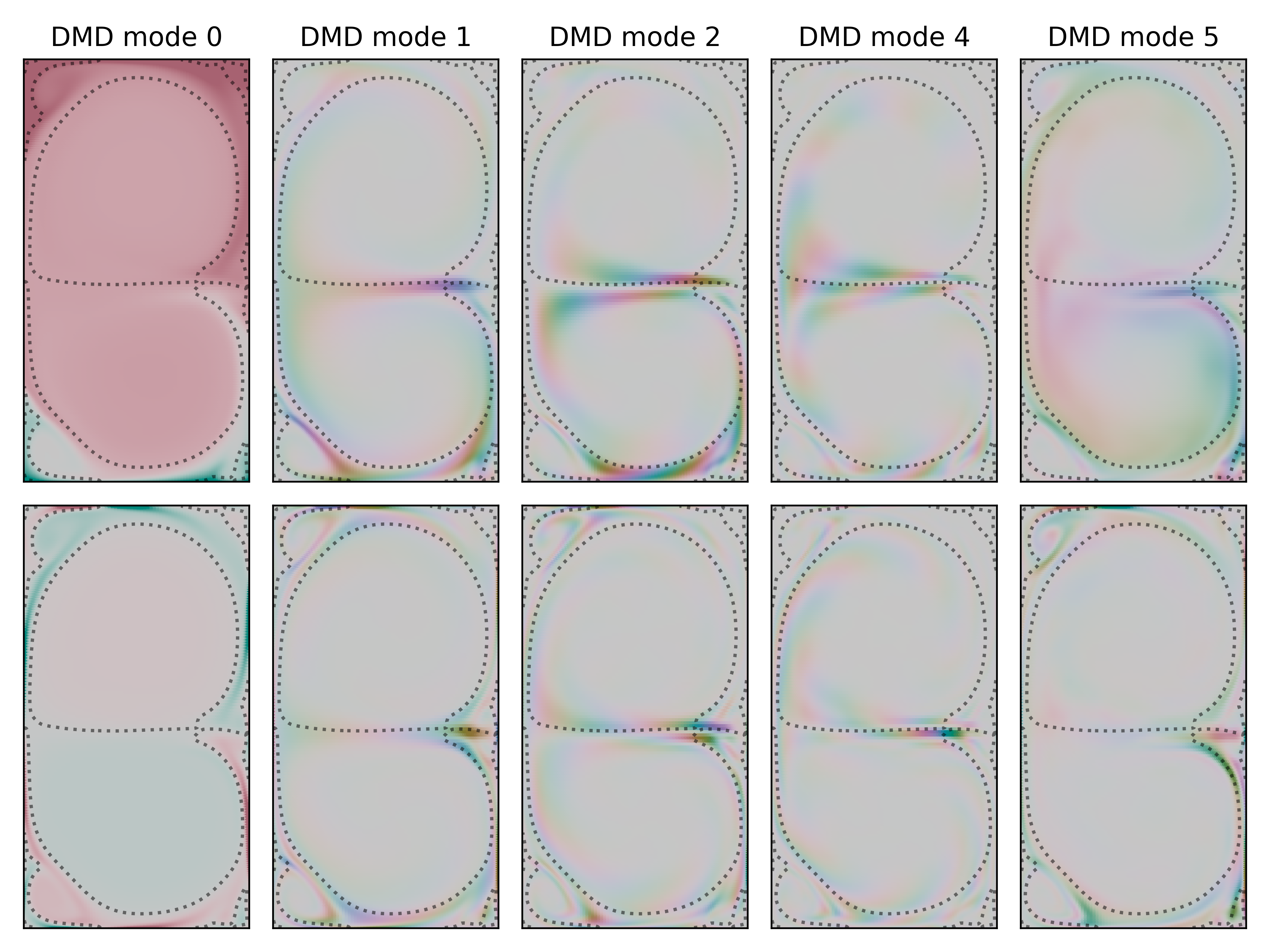}
	\caption{Leading DMD modes of the Rayleigh--B\'enard system computed in Section~\ref{s:NumericsRB}, plotted as average configurations of the flow (top: temperature; bottom: vorticity). Saturation corresponds to eigenfunction modulus, and hue to angle. Lines of zero vorticity are plotted for comparison (grey dots). }
	\label{f:RBmodes}
\end{figure}

	\section{Conclusion}\label{s:Conclusion}
	
	In this paper we have proposed a sampling pseudospectrum $P(\lambda)$, with an estimator $\hat P(\lambda)$ that can be computed from data, and is furnished with a statistical test for eigenvalues. These can be respectively used to find the location of finite-data eigenvalues given infinite-data information, and the infinite-data eigenvalues given a finite sample. We have seen in examples that our indicators do this relatively efficiently: they give estimates largely within a factor of two of the true log-likelihood, which is an excellent outcome given that the unpredictability of the conditioning of the operators we study.
	
	Furthermore, our pseudospectra appear to be useful in identifying which finite-data eigenvectors are good representations of the true ones (Sections~\ref{s:NumericsL63} and~\ref{s:NumericsRB}). This may be related to the fact that the joint eigenspace of infinite-data eigenvalues in $E'$ is given as the range of the projector
	\[ \Pi_{E'} = \int_{\partial E'} \Cl^{-1} \d \lambda, \]
	the finite-data analogue is 
	\[\hat \Pi_{E'} = \int_{\partial E'} \hCl^{-1} \d \lambda =  \int_{\partial E'} (I + \Cl^{-1} \tCl)^{-1} \Cl^{-1}  \d \lambda, \]
	noting that $P_\sym(\lambda)$ measures the likelihood that  $\|\Cl^{-1} \tCl\| = 1$ in some norm (see Proposition~\ref{p:BernsteinBound}). Understanding to what extent this can be formulated rigorously would be an important future direction.
	
	Crucially for applications, we have efficient numerical methods to compute these functions in the same order of time it takes to compute the matrices themselves (Sections~\ref{s:Algorithms} and~\ref{s:Variance}). These use the cone structure of positive definite matrices to reformulate the max--min definition of the pseudospectrum \eqref{eq:Pdef} as the spectral radius of a linear operator on matrices.
	
	We have implemented the methods in this paper in the Julia package SamplingPseudospectrum.jl \cite{SP}.

	The ideas in this paper are very general. They can used to solve a large range of generalised eigenvalue problems. In particular, it would be possible to include Tikhonov-style regularisation terms, which would allow $M < N$, i.e. more observables than data points. By expanding from a vector space of dimension $N$ to RKHS, quantification of operators estimated from kernels, such as in kernel EDMD \cite{Williams15}, would then also be possible. Other kinds of problems this could be applied to include data sampled irregularly in time, by studying characteristic functions of the form
		\[ \Cl[\omega] = \pX(\omega) \left( e^{\lambda \Delta t(\omega)} \pX(\omega)^* - \pY(\omega)^*\right). \]

\subsection*{Acknowledgements}

This work was supported by the Australian Research Council Discovery Early Career Award DE260101080.

\appendix

\section{Rule of thumb from Section \ref{s:ProbabilisticBounds}}\label{a:RuleOfThumb}

We will try and estimate $R_{\partial F}$ under the assumption that the size of $\cl{\omega}$ is relatively uniform. If so, the essential supremum in the definition \eqref{eq:RpartialF} of $R_{\partial F}$ is of the same order as a corresponding $L^2$ estimate:
\[ 	R_{\partial F}^2 = \sup_{\lambda \in \partial F} \inf_{\substack{Q\succ 0\\ P_\sym(\lambda,Q) \geq P^*}} \esssup_{\omega \in \Omega} \| \Cl^{-1} \cl{\omega} - I \|_Q^2 \approx \sup_\lambda \inf_{P_\sym(\lambda,Q)\geq P^*} \bE[\| \Cl^{-1} \cl{\omega} - I \|_Q^2] =: R_{L^2}\]

Expanding the definition of the $\| \cdot \|_Q$ operator norm, we get
\begin{align*}
	\bE[\| \Cl^{-1} \cl{\omega} - I \|_Q^2] = \bE[ \| Q^{-1/2} (\Cl^{-1} \cl{\omega} - I)^{-1} Q (\Cl^{-1} \cl{\omega} - I) Q^{-1/2} \|]\end{align*}
We can bound the standard operator norm by the trace (in fact, for rank-one $\cl{\omega}$ this is an equality), giving 
\begin{align*}
	\bE[\| \Cl^{-1} \cl{\omega} - I \|_Q^2] &\leq \bE[ \tr Q^{-1/2} (\Cl^{-1} \cl{\omega} - I)^{-1} Q (\Cl^{-1} \cl{\omega} - I) Q^{-1/2} ]
\end{align*}
By the independence assumed in Theorem~\ref{t:EigsInCurve}, this is equal to
\begin{align*} \tr Q^{-1/2} \S_\lambda[Q] Q^{-1/2} 
	&= \sum_{n=1}^N \bv_n^* Q^{-1/2} \S_\lambda[Q] Q^{-1/2} \bv_n \\
	&= \sum_{n=1}^N \mu_n^{-1} \bv_n^* \S_\lambda[Q] \bv_n \\
\end{align*}
for $(\bv_n, \mu_n)$ an orthonormal eigenbasis of $Q$.

Now the definition of $P(\lambda,Q)$ \eqref{eq:Pdef}, we have that 
\[ \bv_n^* \S_\lambda[Q] \bv_n \leq P(\lambda,Q)^{-1} \bv_n^* Q \bv_n = P(\lambda,Q)^{-1} \mu_n, \]
so
\[ \tr Q^{-1/2} \S_\lambda[Q] Q^{-1/2}  \leq \sum_{n=1}^N P(\lambda,Q)^{-1} = \frac N{P(\lambda,Q)}. \]
This means that 
\[ R_{\partial F} \approx R_{L^2} \leq \sup_{\lambda \in \partial F} \inf_{\substack{Q\succ 0\\ P_\sym(\lambda,Q) \geq P^*}} \sqrt{\frac{N}{P(\lambda,Q)}} \leq \sqrt{\frac{N}{P_*}},\]
recalling that $P_\sym(\lambda,Q)\leq P(\lambda,Q)$.

As a result, under some common we (non-rigorously) expect the correction in \eqref{eq:EigsInCurveLimit} $ 1 + P^* R_{\partial F}/3$ not to dominate $M P^*/2$ provided 
\[ P^* R_{\partial F} \approx \sqrt{P^* N} \ll M P^*/2,\]
i.e. $M \gg 2\sqrt{N/P^*}$.

\section{Proof of Theorem~\ref{t:EigsInCurve}} \label{a:EigsInCurve}
The proof of Theorem~\ref{t:EigsInCurve} is based on making sure that $\| \Cl^{-1} \tCl \|_Q$ is smaller than $1$ along the boundary of $F$ and therefore no eigenvalues can cross into or out of $F$ (Proposition~\ref{p:CountingEigs}). To do this, it uses a matrix Bernstein bound for fixed $\lambda, Q$ (Proposition~\ref{p:BernsteinBound}) and uses a union bound on a net of $\lambda$ plus some continuity results (Lemmas~\ref{l:PsymPerturbation} and \ref{l:FiniteNumberofQ}) to extend this more broadly.

The first lemma gives us a pseudospectral counting bound:
\begin{proposition}\label{p:CountingEigs}
	Suppose that for all $\lambda \in \partial F$ there exists $Q$ such that $\| \Cl^{-1} \tCl \|_Q < 1$. Then, $\#(\Sigma,F) = \#(\hat \Sigma,F)$.
\end{proposition}
\begin{proof}
	Because $\lambda \mapsto \cl{\omega}$ is analytic and uniformly bounded on $E$, so is $\lambda \mapsto \Cl$, and so therefore is $d_t(\lambda) := \det ((1-t) \Cl + t \hCl)$ for any $t \in [0,1]$. Now, on $\partial F$,  $(1-t) \Cl + t \hCl = \Cl (I  + t\Cl^{-1} \tCl)$. Since $P^* > 0$ on $\partial E$, by Proposition~\ref{p:PsymEigval} $\Cl$ is invertible on $\partial E$. This means that $d_t(\lambda) \neq 0$ if $(I  + t\Cl^{-1} \tCl)$ is invertible. This occurs for $|t|<1$ if $\| \Cl^{-1} \tCl\|_Q  < 1$, which is the case by assumption. 
	
	So for all $t \in [0,1]$, $d_t(\lambda) \neq 0$ on $\partial F$, and therefore by the argument principle under deformation, the number of zeros of $d_0(\lambda) = \det \Cl$ in $F$ is the same as that of $d_1(\lambda) = \det \hCl$. These zeros correspond respectively to the sets $\Sigma$, $\hat \Sigma$ counting algebraic multiplicity, proving the result. 
\end{proof}

\begin{lemma}\label{l:PsymPerturbation}
	For all $Q \succ 0$, $\lambda \mapsto P_\sym(\lambda, Q)$ and $\lambda \mapsto P(\lambda,Q)$ are continuous on $E \backslash \Sigma$.
\end{lemma}
\begin{proof}
	On this set we have $P_\sym(\lambda, Q) > 0$ from Proposition~\ref{p:PsymEigval}.  We can therefore consider $\lambda \mapsto P_\sym(\lambda, Q)^{-1}$ and its continuity. This is
	\begin{align*} P_\sym(\lambda,Q)^{-1} &= \max \left\{  \sup_{\bv \neq \bm{0}} \frac{\bv^* \S_\lambda[Q] \bv}{\bv^* Q \bv}, \sup_{\bv \neq \bm{0}} \frac{\bv^* \mathcal{S_\lambda^*}[Q^{-1}] \bv}{\bv^* Q^{-1} \bv} \right\} 		&= \max \left\{ \|Q^{-1/2} \S_\lambda[Q] Q^{-1/2}\|, \|Q^{1/2} \S_\lambda^*[Q^{-1}] Q^{1/2}\|\right\} \end{align*}
	So all that is necessary is to show that both of the functions of $\lambda \in F$ in the maximum are continuous.
	
	Now, by Proposition~\ref{p:SAsCorrelationSum} and the fact that the $\cl{\theta^m\omega}$ are independent,
	\[ Q^{-1/2} \S_\lambda[Q] Q^{-1/2} = Q^{-1/2} \bE\left[(\Cl^{-1} \cl{\omega})^* Q \Cl^{-1} \cl{\omega}\right] Q^{-1/2}  \]
	Now, by assumption $D_\lambda(\omega)$ and $\Cl$ are analytic and uniformly bounded on $E$, so they are continuous uniformly over $\omega$ in $\partial F$, a compact subset of $\interior E$. Since $\Cl$ is invertible for $\lambda \in \partial F$ (and $\omega$-independent), so is $\Cl^{-1}$. This makes that the expression inside the expectation is continuous on $\lambda \in F$.
	
	The same argument gives the corresponding result for $\|Q^{1/2} \S_\lambda^*[Q^{-1}] Q^{1/2}\|$, and a small simplification of it gives us $P(\lambda,Q)$, and we are done.
\end{proof}

\begin{lemma}\label{l:FiniteNumberofQ}
	The following are true:
	\begin{itemize}
		\item There exist a finite number of matrices $Q_k \succ 0$ such that for every $\lambda \in \partial F$, $P_\sym(\lambda,Q_k) \geq P^*$ for some $Q_k$.
		\item $R_{\partial F} < \infty$.
		\item There exists a constant $\Delta P$ such that for all $\lambda \in \partial F$ there is a $Q_k$ such that for $|\lambda' - \lambda| < \epsilon_0$, $\esssup_{\omega \in \Omega} \left\| \Cl^{-1} \cl{\omega} - \Cl[\lambda']^{-1} \cl[\lambda']{\omega} \right\|_{Q_k} < \Delta P |\lambda' - \lambda|$.
	\end{itemize}
\end{lemma}
\begin{proof}
	Lemma~\ref{l:PsymPerturbation} combined with continuity of $\Cl, \cl{\omega}$ in $\lambda$ gives that $P_\sym(\lambda, Q)$ is continuous in $\lambda$. We thus get an open cover of $\partial F$ by sets $\{ \lambda : P_\sym(\lambda, Q) < P^* \}$ for $Q \succ 0$, which we can restrict to a finite subcover, giving the first part.
	
	The second part is a consequence of the first. We can restrict to take the $Q$ infimum over a finite number of $Q_k$, meaning the conversion between $\| \cdot \|$ and $\| \cdot \|_{Q'}$ has finite constants. Then, we have by assumption on $\cl{\omega}$ that it is uniformly bounded on $\partial F$: because $C_\lambda$ is continuous and invertible on $\partial F$ the norm of its inverse is uniformly bounded.
	
	The final part is a consequence of the analyticity of $\Cl^{-1}, \cl{\omega}$ on a neighbourhood of $\partial F$: $\Cl^{-1} \cl{\omega}$ must be uniformly $\| \cdot \|_{Q_k}$-Lipschitz on $\partial F$ and therefore the bound must hold.
\end{proof}

\begin{proposition}\label{p:BernsteinBound}
	For all $t > 0$ and $\lambda \in \partial F$,
	\[ \bP\left( \| \Cl^{-1} \tCl \|_Q \geq t \right) \leq 2N \exp{-\frac{M t^2/2}{P_\sym(\lambda,Q)^{-1} + t R(\lambda,Q)/3}} \]
	where 
	\[ R(\lambda,Q) = \esssup_{\omega \in \Omega}  \left\| \Cl^{-1} \cl{\omega} \right\|_Q \]
\end{proposition}
\begin{proof} 
	In the following we note that for any matrix $A$, $\|A \|_Q = \| Q^{1/2} A Q^{-1/2} \|$. 
	
	Because the $\cl{\omega}$ are independent, $D_{\theta^m\omega} := Q^{1/2} \Cl^{-1} (\cl{\sigma^m\omega} - \Cl) Q^{-1/2}$ is a sequence of mean-zero independent random variables. Hence, we can apply \cite[Theorem~1.6]{Tropp12} to
	\[ Q^{1/2} \Cl^{-1} \tCl Q^{-1/2} = \sum_{m=1}^M D_{\sigma^m\omega}. \]
	giving us
	\[ \bP\left( \| Q^{1/2} \Cl^{-1} \tCl Q^{-1/2}  \| \geq t\right) \leq 2N \exp{-\frac{t^2/2}{\alpha + \beta t/3}} \]
	where
	\begin{align*}
		\alpha &= \max \left\{ \left\|\bE[ \sum_{m=1}^M  D_{\sigma^m\omega}^* D_{\sigma^m\omega} \right\|,
		\left\| \sum_{m=1}^M \bE[ D_{\sigma^m\omega} D_{\sigma^m\omega}^*  \right\| 
		\right\} \\ 
		&=M \max \left\{  \sup_{\bv \neq \bm{0}} \frac{\bv^* \S_\lambda[Q] \bv}{\bv^* Q \bv}, \sup_{\bv \neq \bm{0}} \frac{\bv^* \mathcal{S_\lambda^*}[Q^{-1}] \bv}{\bv^* Q^{-1} \bv} \right\}= M P_{\sym}(\lambda,Q)^{-1} \\
		\beta &= \sum_{m=1}^M \esssup_{\omega \in \Omega}  \| D_{\sigma^m\omega}  \| = M R(\lambda, Q) \end{align*}
	
	This gives us the required result.
\end{proof}

\begin{proof}[Proof of Theorem~\ref{t:EigsInCurve}]
	By Proposition~\ref{p:CountingEigs}, we can lower-bound eigenvalue-matching probabilities by looking at matrix norms for $\lambda \in \partial F$: 
	\[ \bP\left(\#(\Sigma,F) = \#(\hat \Sigma,F)\right) \geq \bP\left( \forall \lambda \in \partial F\ \exists Q: \| \Cl[\lambda]^{-1} \tCl[\lambda] \|_Q < 1\right). \]
	This gives an upper bound on the complementary event:
	\[ \bP\left(\#(\Sigma,F) \neq \#(\hat \Sigma,F)\right) \leq \bP\left( \exists \lambda \in \partial F\ \forall Q: \| \Cl[\lambda]^{-1} \tCl[\lambda] \|_Q > 1\right) \]
	
	Since $\partial F$ is a rectifiable curve, for any $\epsilon < 1$ we can find $J \leq C/\epsilon$ points $\{ \lambda_j \}$ in $\partial F$ whose $\epsilon$-neighbourhoods cover $\partial F$.
	
	For each $\lambda_j$ let $Q_{k(j)}$ be one of the finite number of matrices from Lemma~\ref{l:FiniteNumberofQ} with $P_*(\lambda_j,Q_{k(j)}) > P_*$. Then, Proposition~\ref{p:BernsteinBound} gives us that for $t \leq 1$,
	\begin{align*} \bP\left( \| \Cl[\lambda_j]^{-1} \tCl[\lambda_j] \|_{Q_{k(j)}} \geq t \right) &\leq 2N \exp{-\frac{M t^2/2}{P_\sym(\lambda_j,Q_{k(j)})^{-1} + t R(\lambda_j,Q_{k(j)})/3}} \\
		&\leq 2N \exp{-\frac{M t^2/2}{P_*^{-1} + t R_{\partial F}}} \\
		&\leq 2N \exp{-\frac{M (2t-3)/2}{P_*^{-1} + R_{\partial F}}}\end{align*}
	
	We can make a union bound over all our $\lambda_j$:
	\[ \bP\left( \exists j :  \| \Cl[\lambda_j]^{-1} \tCl[\lambda_j] \|_{Q_{k(j)}} \geq t \right) \leq \frac{2CN}{\epsilon}  \exp{-\frac{M (2t-3)/2}{P_*^{-1} + R_{\partial F}}}\]

	Now, any given $\lambda \in \partial F$ is less than $\epsilon$ from some $\lambda_j$. From Lemma~\ref{l:PsymPerturbation} and the triangle inequality, it then can only have $\| \Cl[\lambda]^{-1} \tCl[\lambda] \|_{Q_{k(j)}} \geq 1$ if  $\| \Cl[\lambda_j]^{-1} \tCl[\lambda_j] \|_{Q_{k(j)}} \geq 1 - \Delta P \epsilon$. Thus,
	\begin{align*}
		\bP\left( \exists \lambda \in \partial F\ \forall Q: \| \Cl[\lambda]^{-1} \tCl[\lambda] \|_Q > 1\right) &\leq
		\bP\left( \exists \lambda \in \partial F:  \| \Cl[\lambda]^{-1} \tCl[\lambda] \|_{Q_{k(j)}} \geq 1 \right)\\
		&\leq \bP\left( \exists j:  \| \Cl[\lambda_j]^{-1} \tCl[\lambda_j] \|_{Q_{k(j)}} \geq 1  - \Delta P \epsilon \right)\\
		& \leq \frac{2CN}{\epsilon}  \exp{-\frac{M (1 -  2\Delta P \epsilon)/2}{P_*^{-1} + R_{\partial F}}}.\end{align*}
	Choosing $\epsilon = C'/M$ and increasing our constant, we arrive at the required result.
\end{proof}

\section{Proof of Theorem~\ref{t:StatTest}}\label{a:StatTest}
The proof of both parts of Theorem~\ref{s:StatisticalTest} reduce to studying convergence of $\hat \S_\lambda$ as the data $M \to \infty$. For the first part, this limit is well-defined as $\S_\lambda$. In the second part, $\hCl$ converges to the non-invertible $\Cl$. We therefore need to study of how $\hCl$ acts on the kernel of $\Cl$.
 
\begin{proof}[Proof of Theorem~\ref{t:StatTest}]
	Suppose we are not at an eigenvalue of $\Cl$. Then $\hat P(\lambda)$ is the spectral radius of the matrix operator
	\[  \hat \S_\lambda: Q \mapsto \hat \V[(\hCl^{-1})^* Q \hCl^{-1}]. \]
	As $M\to\infty$ almost surely, $\hCl \to \Cl$ so $\hCl^{-1} \to \Cl^{-1}$. Similarly, $\hat V \to V$ almost surely. As a result, $\hat \S_\lambda[Q]$ converges to $\S_\lambda$ almost surely, and therefore so too do the spectral radii, which give $\hat P$ and $P$ from Theorem~\ref{t:PIsSpecRad}: $1/\hat P(\lambda) \to 1/P(\lambda)$, as required.
	\\

	Now suppose instead that $\lambda \in \Sigma$ and $\Cl$ has a nullity of $1$. To describe the nullspaces, let us choose unit vectors $\br \in \ker\Cl$ and $\bl \in \ker\Cl^*$.
	
	Although $\S_\lambda: Q \mapsto \V[(\Cl^{-1})^* Q \Cl^{-1}]$ is not well-defined because $\Cl$ has non-invertible kernel, we can think of it as having an ``eigenmatrix'' $\Qst = \V_\lambda[\bl \bl^*]$. We will try and lower-bound $\hat P(\lambda)$ using this matrix. 
	
	\newcommand{\hbc}{\hat{\bm{c}}}
	\newcommand{\hbb}{\hat{\bm{b}}}
	
	Let us make orthogonal decompositions of the domain of $\Cl, \tilde\Cl$ as $\C^N = \im \Cl^* \oplus \langle r \rangle$ and their codomains as $\C^N = \im \Cl \oplus \langle l \rangle$. This gives a decomposition
	\[ \Cl = \begin{pmatrix}
		A & \bm{0} \\ \bm{0}^* & 0
	\end{pmatrix},\, \hCl = \begin{pmatrix}
		\hat A & \hat{\bm{b}} \\ \hbc^* & \epsilon
	\end{pmatrix}. \]
	Of importance here are three facts: $A$ has trivial kernel and is therefore invertible; $\hat A - A,\, \hbb,\, \hbc$ and $ \epsilon$ all go to zero almost surely as $M\to\infty$; and $\epsilon = \bl^* \tCl \br$.
	
	Using the Schur complemen $\epsilon' := \epsilon - \hbc^* \hat A^{-1} \hbb$, we can write 
	\[ \hCl^{-1} = \frac{1}{\epsilon'} \left( \begin{pmatrix} \epsilon' \hat A^{-1} & \bm{0}^* \\ \bm{0} & 0 \end{pmatrix} + \begin{pmatrix} - \hat A^{-1} \hbb \\ 1 \end{pmatrix} \begin{pmatrix} - \hbc^* \hat A^{-1} & 1 \end{pmatrix} \right) \]
	In the limit, this means , 
	\begin{align} \epsilon' \hCl^{-1} \to \begin{pmatrix} 0 & 0 \\ 0 & 1 \end{pmatrix} = \br \bl^* \label{eq:Ctorank1}\end{align}
	almost surely as $M \to \infty$. 
	
	Consequently, for any $Q \succeq 0$,
	\[ |\epsilon'|^2 \hat{S}_\lambda[Q] = M^{-1} \hat\V[(\epsilon' \hCl^{-1})^* Q \epsilon' \hCl^{-1}] \]
	We can then use \eqref{eq:Ctorank1} to get that
	\[ \lim_{M\to\infty} (\epsilon' \hCl^{-1})^* Q \epsilon' \hCl^{-1} = \bl \br^* Q \br \bl^* = r[Q] \bl \bl^*, \]
	where $r[Q] := \br^* Q \br \geq 0$ for $Q \succeq 0$.
	
	Noting that because of its asymptotic consistency, $A \mapsto \hat\V[A]$ is uniformly bounded in $M$, we find that
	\[ \lim_{M\to\infty} \hat\V_\lambda\left[(\epsilon' \hCl^{-1})^* Q \epsilon' \hCl^{-1}\right] = r[Q]  \lim_{M\to\infty} \hat\V\left[ \bl \bl^*\right] \textrm{ a.s.} \]
	which, using the definitions of $\hat\S_\lambda$ and $\Qst$ gives that 
	\[ \lim_{M\to\infty} |\epsilon'|^2 \hat\S_\lambda[Q] = r[Q] \Qst \textrm{ a.s.} \]
	In other words, the matrix operator $|\epsilon'|^2 \hat\S_\lambda$ converges to the rank-one matrix operator $\Qst r: Q \mapsto r[Q] \Qst$ almost surely, and consequently the spectral radius of $|\epsilon'|^2 \hat\S_\lambda$ must converge accordingly. As a result,
	\begin{equation} \lim_{M\to\infty} \frac{M |\epsilon'|^2}{M \hat P(\lambda)} = \specrad \Qst r = r[\Qst] =: \sigma^2_\epsilon. \label{eq:Epsdashlimit}\end{equation}
	
	Our question about $M \hat P(\lambda)$ therefore reduces to understanding $M |\epsilon'|^2$ in relation to $\sigma^2_\epsilon$. We can start by reducing from $\epsilon'$ to $\epsilon$ using that
	\[ \sqrt{M} \epsilon' = \sqrt{M} \epsilon - \sqrt{M} \hbc^* \hat A^{-1} \hbb.\] 
	To deal with the difference we note that $\hbc^*$ obeys a central limit theorem so $\sqrt{M} \hbc^*$ converges in distribution; on the other hand, $\hat A^{-1} \to A$ and $\hbb \to \bm{0}$ almost surely, so by Slutsky's theorem, $\sqrt{M} \hbc^* \hat A^{-1} \hbb$ converges in distribution to the constant 0, and hence also in probability. Hence, $M |\epsilon'|^2$ converges in probability to $M |\epsilon|^2$.
	
	If $M|\epsilon|^2$ converges in law to some distribution $\sigma^2_\epsilon D$, so too will $M|\epsilon'|^2$, giving us that
	\begin{align} \lim_{M\to\infty} \bP\left( M \hat P(\lambda) > c \right) = \bP(D > c). \label{eq:Phatvseps} \end{align}
	
	We then must ask, does this occur, and what is the limiting distribution $D$? We have that $\epsilon = \bl^* \tCl \br$ so
	\[ \lim_{M\to\infty} \bE[M |\epsilon|^2] = \br^* \left( \lim_{M\to\infty} M \bE[\tCl^* \bl \bl^* \tCl] \right) \br = \br^* \Qst \br = \sigma^2_\epsilon. \]
	
	If $\sigma^2_\epsilon = 0$ then $M |\epsilon|^2 \to_{\ell} 0$, so we could choose any distribution we wanted for $D$.
	
	If $\sigma^2_\epsilon > 0$, then by assumption $\sqrt{M} \epsilon$ converges to a normal distribution, but it may be complex-valued so that $\sqrt{M} \epsilon / \sigma_\epsilon \to_d \alpha + i \beta$ with $\alpha,\beta$ centred {\it real} normal variables with $\bE[\alpha^2 + \beta^2] = 1$. 
	Now, $\alpha, \beta$ may be correlated, but nevertheless we find that
	$D \sim \alpha^2 + \beta^2$ is a Gaussian quadratic form with expectation $\bE[D] = 1$, and rank at most $2$. From \cite{Szekely03} we can bound
	\[ \bP(D > c) \leq \max\left\{\bP(\chi_1^2 > c), \bP(\tfrac{\chi_2^2}{2} > c)\right\}. \]
	
	Combining this with \eqref{eq:Phatvseps} the theorem follows.
\end{proof}

\section{Proof of Theorem~\ref{t:PIsSpecRad}}\label{a:PIsSpecRad}
	
	\begin{proof}[Proof of Theorem~\ref{t:PIsSpecRad}]
		We begin with part \ref{tt:PIsSpecRad-cone}. If $Q \succeq 0$, then we must have  $(L^{-1})^*  Q L^{-1} \succeq 0$, and therefore by assumption on $\V$,  $\S[Q] \succeq 0$. Because the set of positive semi-definite operators is a proper cone as in \cite{Berman94}, $\S_\lambda$ must have various properties posessed by cone-preserving operators
		. In particular, the rest of \ref{tt:PIsSpecRad-cone} is a consequence of \cite[Theorem~3.2]{Berman94}.
		
		For part \ref{tt:PIsSpecRad-specrad}, we will reformulate the definition of $P$. Letting $\bv = L^{-1} \bw$, we have that 
		\begin{align} P^{-1} = \inf_{Q \succ 0} \sup_{\bw \ne 0} \frac{\bw^* (L^{-1})^* \V[Q] L^{-1} \bw}{\bw^* Q \bw} = \inf_{Q \succ 0} \sup_{\bw \ne 0} \frac{\bw^* \S[Q] \bw}{\bw^* Q \bw}.\label{eq:PIsRayleigh} \end{align}
		We can rephrase 
		\[ \sup_{\bw \ne 0} \frac{\bw^* \S[Q] \bw}{\bw^* Q \bw} = \inf\{ \mu \mid \forall \bw :  \bw^* \S[Q] \bw \leq \bw^* \mu Q \bw\} = \inf\{ \mu : \S[Q] \preceq \mu Q \}, \]
		and so 
		\[ P^{-1} = \inf_{Q\succ 0} \inf\{ \mu : \S[Q] \preceq \mu Q \} = \inf \{ \mu \mid \exists Q\succ 0 : \S[Q] \preceq \mu Q \} \]
		which by \cite[Theorem~1.1]{Akian11} is precisely the spectral radius of $\S$.
		
		%
		
		We now move to \ref{tt:PIsSpecRad-bounds}. $\S$ being primitive implies it is irreducible \cite[Corollary 4.17]{Berman94}, and so has a simple positive eigenvalue at $1/P$ whose eigenvector $Q_\lambda$ is positive definite. From \cite[Theorem 4.10]{Berman94}, $\S$ being primitive implies all other eigenvalues are strictly smaller in magnitude. Consequently, for all $Q \succeq 0$ we have $P^k \S^k[Q] \to \alpha Q_\lambda$ exponentially fast for some constant $\alpha > 0$. 
		
		Now, because that the generalised spectrum $\sigma_Q$ is the regular spectrum of $\S[Q]^{-1/2} Q \S[Q]^{-1/2}$, the functions
		\begin{align*} \underline{m}(Q) &= \min \sigma_Q,& \overline{m}(Q) &= \max \sigma_Q \end{align*}
		are locally Lipschitz when $\S[Q] \succ 0$, and $\underline{m}(Q_\lambda) = \overline{m}(Q_\lambda) = P$. 
		
		Putting this together, we have that $\underline{m}(\S^k[Q]) \to \underline{m}(Q_\lambda) = P$ exponentially fast, and similarly for $\overline{m}(\S^k[Q])$, as required.
	\end{proof}

\section{Proof of variance operator results}\label{a:VPropositions}

In this appendix we prove Propositions~\ref{p:VAsCorrelationSum} and~\ref{p:VhatConstruction}.
\begin{proof}[Proof of Proposition~\ref{p:VAsCorrelationSum}]
	We have for any Hermitian $H$,
	\[ \V_\lambda[H] =\lim_{M\to\infty}  M \bE[\tCl^* H \tCl]\]
	so $\V_\lambda$ is linear on Hermitian matrices, as the limit of linear operators on Hermitian matrices. 
	
	Furthermore, for any $Q \succeq 0$, 
	\[ \bv^* \V_\lambda[Q] \bv =\lim_{M\to\infty}  M \bE[\| \tCl \bv \|^2_Q] \geq 0\]
	so $\V_\lambda$ preserves positive semi-definite matrices.
	
	We can decompose any Hermitian $H$ as a sum $H = \sum_{n=1}^N \alpha_n \bw_n \bw_n^*$ for appropriate vectors $\bw_n \in \C^N$ and scalars $\alpha_n \in \R$. Then for any $\bv \in \C^N$, let
	\[ \bv^* \V_\lambda[H] \bv = \lim_{M\to\infty} M \sum_{m=1}^M \bE\left[ (\tCl \bv)^* \sum_{n=1}^N \alpha_n \bw_n \bw_n^* \tCl \bv \right] = \sum_{n=1}^N \alpha_n \lim_{M\to\infty} M \bE[|S_M(\omega,\varphi_n)|^2]  \]
	where $\varphi_n(A) = \bw_n^* (A - \Cl) \bv$. By Assumption~\ref{as:birkhoff}, the limits must converge for any $\bv, H$, giving us that $\V_\lambda$ is a well-defined linear operator on matrices. By Assumption~\ref{as:birkhoff} it must also be equal to
	\begin{align*} &\sum_{n=1}^N \alpha_n \Ces \sum_{t = -\infty}^\infty \bE[\varphi_n(\cl{\omega})^* \varphi_n(\cl{\theta^t \omega})]\\
		&\quad =\Ces \sum_{t = -\infty}^\infty   \bE\left[\bv^* (\cl{\omega} - \Cl)^* \left(\sum_{n=1}^N \alpha_n \bw_n^* \bw_n\right) (\cl{\theta^n\omega} - \Cl) \bv\right] \\
		&\quad =\Ces \sum_{t = -\infty}^\infty \bv^* \bE\left[(\cl{\theta^n\omega} - \Cl)^* H (\cl{\theta^n\omega} - \Cl) \right] \bv \\
		&\quad =\bv^* \Ces \sum_{t = -\infty}^\infty \left( \bE[\cl{\theta^n\omega}^* H \cl{\theta^n\omega}] - \Cl^* H \Cl\right) \bv
	\end{align*}
	as required.
\end{proof}

\begin{proof}[Proof of Proposition~\ref{p:VhatConstruction}]
	We need to check the three stipulations of Assumption~\ref{as:Vhat}. Linearity is immediate.
	
	Positive definiteness arises as follows. The Fourier series \[\widehat{\kappa_M}(\xi) = \widehat{\kappa_p}(\xi)\cdot L_M \widehat{\kappa_w}(L_M\xi) \geq 0.\] 
	As a result, the bi-infinite Hermitian Toeplitz matrix $(\kappa_M(j-k))_{j,k \in \Z}$ is positive semi-definite, so its truncation to an $M\times M$ matrix $K' = (\kappa_M(j-k))_{j,k = 1,\ldots, M}$ must also be positive semi-definite with square root $K = \sqrt{K'}$.
	Then, for any $\bv \in \C^N$ and $Q \succeq 0$, we have
	\begin{align*} 
		\bv^* \hat \V_\lambda[Q] \bv &= \frac{1}{M} \sum_{m,m' = 1}^M \sum_{k=1}^M K_{m,k} K_{k,m'} ((\cl{\omega_m} - \hCl)\bv)^* Q (\cl{\omega_{m'}} - \hCl) \bv\\
		&= \sum_{k=1}^M \frac{1}{M} \left\|\sum_{m=1}^M K_{k,m} (\cl{\omega_m} - \hCl)\bv\right\|_Q^2
		\geq 0 
	\end{align*}
	so $\hat \V_\lambda[Q]$ is positive definite.

	For asymptotic consistency, we need that as $M\to\infty$ the variance of $\hat\V_\lambda[Q]$ goes to zero and expectation of $\hat\V_\lambda[Q]$ converges to $\V_\lambda[Q]$. 	Bounding the variance by adding standard deviations, we have that 
	\[ \sqrt{\bV[\hat\V_\lambda[Q]]} \leq \sum_{\ell = -(L' + L_M)}^{L' + L_M} |\kappa_M(\ell)| \sqrt{\bV[\hat\Gamma_\ell]} \leq (2L_M - 1) \| \kappa_M \|_\infty \sqrt{\frac{C}{M}} \]
	as $M \to \infty$. From its assumptions, $\kappa_M$ is bounded, and using that $L_M = o(\sqrt{M})$ we get that the variance goes to zero. 
	
	Now we must consider the expectation. For simplicity we consider the case $\ell \geq 0$, the $\ell < 0$ case is similar. We wish to show that 
	\begin{equation} \lim_{M\to\infty}\sum_{\ell=-(L'+L_M)}^{L' + L_M} \kappa_M(\ell) \bE[\hat\Gamma_\ell] = \V_\lambda[Q]. \label{eq:ExpDesiderandum}\end{equation}
	Using $\hCl = \Cl + \tCl$, we can rewrite our expectations as 
	\begin{equation} \bE[\hat\Gamma_\ell] = \frac{1}{M} \sum_{m=1}^{M-\ell} \bE\left[\left((\cl{\omega_m}-\Cl) - \tCl\right)^* Q \left((\cl{\omega_{m+\ell}}-\Cl)- \tCl\right)\right] = T_{1,\ell,M} + T_{2,\ell,M} + T_{2,-\ell,M}^* + T_{3,\ell,M}  \end{equation}
	where the latter terms come from expanding the large brackets. The first term is 
	\[ T_{1,\ell,M} =\frac{1}{M} \sum_{m=1}^{M-\ell} \bE\left[\left(\cl{\omega_m}-\Cl\right)^* Q \left(\cl{\omega_{m+\ell}}-\Cl\right) \right] = \frac{M-\ell}{M} \Gamma_k. \]
	
	Considering only this term in \eqref{eq:ExpDesiderandum} 
	\[  \sum_{\ell=-(L' + L_M)}^{L' + L_M} \kappa_M(\ell) \bE[T_{1,\ell,M}] = \sum_{\ell=-(L' + L_M)}^{L' + L_M} \frac{M-\ell}{M} \kappa_M(\ell) \Gamma_k  \]
	which converges to $\V_\lambda[Q]$ as $M\to\infty$. It therefore remains to show that the corresponding sums of the other $T_{i,\ell,M}$ converge to zero as $M\to\infty$
	
	The last term
	\[ T_{3,\ell,M} = \frac{1}{M} \sum_{m=1}^{M-\ell} \bE[\tCl^* Q \tCl] = \frac{M-\ell}{M} \bE[\tCl^* Q \tCl]  \]
	which gives
	\begin{align*} \lim_{M\to\infty} \sum_{\ell=-L_M}^{L_M} \kappa_M(\ell) T_{3,\ell,M} &= \lim_{M\to\infty} \sum_{\ell=-(L' + L_M)}^{L' + L_M} \kappa_M(\ell) \frac{M-\ell}{M} \bE[\tCl^* Q \tCl] \\
		&= \lim_{M\to\infty} \frac{1}{L_M }\sum_{\ell=-(L' + L_M)}^{L' + L_M} \kappa_M(\ell) \frac{M-\ell}{M} \cdot \lim_{M\to\infty} L_M \bE[\tCl^* Q \tCl] \\
		&= 2 \cdot 0 = 0, \end{align*}
	since $\sum_\ell \kappa_m = 2L_M + 1$, and $\bE[\tCl^* Q \tCl] = \mathcal{O}(1/M)$.
	
	The middle term is
	\[ T_{2,\ell,M} = \frac{1}{M} \sum_{m=1}^{M-\ell} \bE[(\cl{\omega_{m+\ell}} - \Cl) Q \tCl] = \bE[\tCl^* Q \tCl] - \frac{1}{M} \bE\left[\sum_{m=1}^\ell (\cl{\omega_{m}} - \Cl) Q \tCl\right]. \]
	From the $T_{3,\ell,M}$ argument we already know that $\bE[\tCl^* Q \tCl]$ is uniformly $\mathcal{O}(1/M)$ as $M\to\infty$. As for the other part, working through Assumption~\ref{as:birkhoff} on Birkhoff sums gives us that 
	\[ \bE\left[\left\| \sum_{m=1}^T (\cl{\omega_m} - \Cl) \right\|^2\right] = \mathcal{O}(T),\, T \to \infty \]
	uniformly, so $\bE[\|\sum_{m=1}^\ell \cl{\omega_{m}} - \Cl\|^2]=\mathcal{O}(\ell)$ and $\bE[ \| \tCl \|^2] = \mathcal{O}(1/M)$. By the Cauchy--Schwarz inequality, we therefore get 
	\[ \left\| \bE\left[\sum_{m=1}^\ell(\cl{\omega_{m}} - \Cl) Q \tCl\right] \right\|\leq \mathcal{O}\left(\sqrt{\frac\ell M}\right) \leq \mathcal{O}\left(\sqrt{\frac{L_M}M}\right) = o(1)\]
	and therefore
	\[ \|T_{2,\ell,M}\| \leq \mathcal{O}(1/M) + \frac{1}{M} o(1) = \mathcal{O}(1/M). \]
	This gives us that
	\begin{align*} \lim_{M\to\infty} \left\|\sum_{\ell=-(L' + L_M)}^{L' + L_M} \kappa_M(\ell) T_{2,\ell,M} \right\| &\leq \mathcal{O}((L' + L_M)/M) = 0
	\end{align*}
	giving us the required result.
\end{proof}

\bibliographystyle{plain}
\bibliography{edmd-resolvent}

\end{document}